\documentclass[12pt]{article}
\usepackage{graphicx}
\usepackage{authblk} 
\usepackage{amsmath}
\usepackage{amssymb}
\usepackage{amsthm}
\usepackage{color}
\usepackage{pagecolor}
\usepackage{url}
\usepackage{a4wide}
 \usepackage{youngtab}
\usepackage{hyperref}
\hypersetup{colorlinks=false, linkbordercolor={1 1 1}, citebordercolor={1 1 1}, urlbordercolor={1 1 1}} 
\usepackage[normalem]{ulem}
\usepackage[section]{placeins} 

\newtheorem{theorem}{Theorem}
\newtheorem{proposition}[theorem]{Proposition}

\newtheorem{definition}[theorem]{Definition}

\def\paragraph#1{{\vskip3mm\noindent \bf #1 }}

\newcommand{\Z}{\mathbb Z}
\newcommand{\N}{\mathbb N}
\newcommand{\cA}{\mathcal A}
\newcommand{\cX}{\mathcal X}

\newcommand{\cE}{\mathcal E}

\newcommand{\cQ}{\mathcal Q}

\newcommand{\tth}{\mathtt h}
\newcommand{\ttn}{\mathtt n}
\newcommand{\ttt}{\mathtt t}
\newcommand{\ttN}{\mathtt N}

\newcommand{\tts}{\mathtt s}

\newcommand{\ttts}{\mathtt s^{\diamond}}

\definecolor{lightgray}{rgb}{0.83, 0.83, 0.83}
\definecolor{lavendermist}{rgb}{0.9, 0.9, 0.98}
\definecolor{backgreen}{RGB}{82, 167, 106}

\def\sqr{\vcenter{
         \hrule height.1mm
         \hbox{\vrule width.1mm height2.2mm\kern2.18mm\vrule width.1mm}
         \hrule height.1mm}}                  
\def\square{\ifmmode\sqr\else{$\sqr$\vskip 3mm}\fi}

\newcommand{\nn}{\nonumber}
\newcommand{\vep}{\varepsilon}

\definecolor{cmm}{rgb}{0,.6,0.4}
\definecolor{cmm'}{rgb}{.6,0,.4}

\title{\textsc{Box-ball system: soliton and tree \\decomposition of excursions}}
\date{\today}

\author[1]{Pablo A. Ferrari}
\author[2]{Davide Gabrielli}
\affil[1]{\small \sl Universidad de Buenos Aires, {\tt pferrari@dm.uba.ar}}
\affil[2]{\small \sl Universit\`a di L'Aquila, {\tt gabriell@univaq.it}}

\begin{document}


\maketitle

\begin{abstract}
We review combinatorial properties of solitons of the Box-Ball system introduced by Takahashi and Satsuma in 1990 \cite{TS}. Starting with several definitions of the system, we describe ways to identify solitons and review a proof of the conservation of the solitons under the dynamics. Ferrari, Nguyen, Rolla and Wang 2018 \cite{FNRW} proposed a soliton decomposition of a configuration into a family of vectors, one for each soliton size. Based on this decompositions, the authors \cite{fg18} propose a family of measures on the set of excursions which induces invariant distributions for the Box-Ball System. Furthermore, we propose a new soliton decomposition which is equivalent to a branch decomposition of the tree associated to the excursion, see Le Gall \cite{MR942038}. A ball configuration distributed as independent Bernoulli variables of parameter $\lambda<1/2$ is in correspondence with a simple random walk with negative drift $2\lambda-1$ and infinitely many excursions over the local minima. In this case the soliton decomposition of the walk consists on independent double-infinite vectors of iid geometric random variables \cite{fg18}. We show that this property is shared by the branch decomposition of the excursion trees of the random walk
and discuss a corresponding construction of a Geometric branching process with independent but not identically distributed Geometric random variables.

\bigskip


\smallskip


\end{abstract}
\tableofcontents
\section{Introduction}
\label{S1}
The Ball-Box-System (BBS) is a cellular automaton introduced by Takahashi and Satsuma \cite{TS} describing the deterministic evolution of a finite number of balls on the infinite lattice $\mathbb Z$. A ball configuration $\eta$ is an element of $\{0,1\}^\Z$, where $\eta(i)=1$ indicates that there is a ball at box $i\in\Z$. For ball configurations with a finite number of balls, the dynamics is as follows. A carrier starts with zero load to the left of the occupied boxes, visits successively boxes from left to right and at each box proceeds as follows (a)  if the box is occupied, the carrier increases its load by one and the box becomes empty or (b) if the box is empty and the carrier load is positive, then the carrier load decreases its load by one and the box becomes occupied. This mechanism is illustrated with an example in Figure \ref{basket}.

\begin{figure}
	
	\centering
	
	\includegraphics{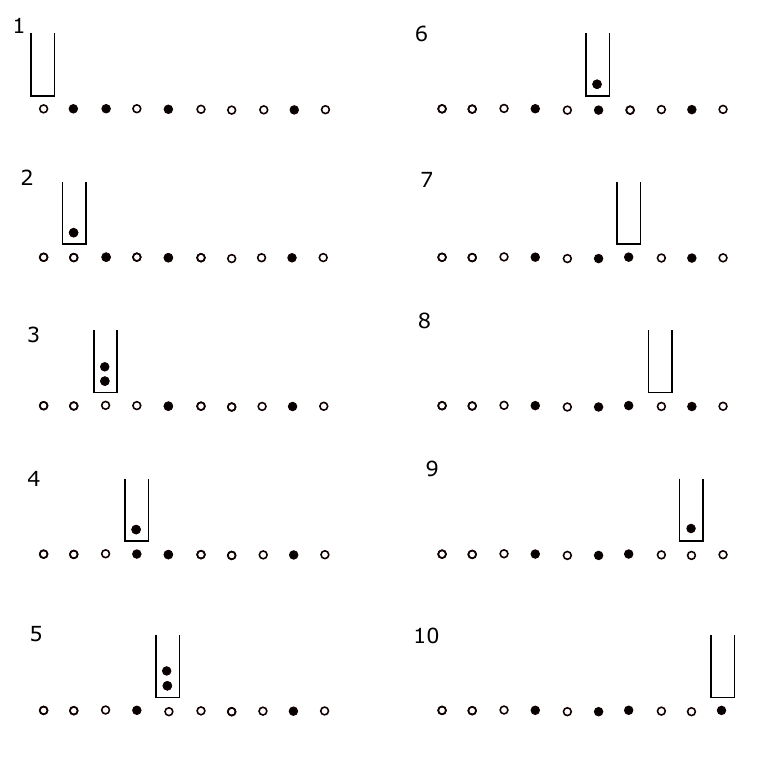}
	
	\caption{An example of the evolution of the Ball-Box-System. The initial configuration of balls $\eta$ is the one in the frame 1 where the boxes not drawn are all empty. The carrier starts empty on the left of the particles and moves to the right following the increasing order of the frames. In the last frame 10 the carrier is empty and will remain empty. The configuration $T\eta$ is the one drawn in the frame 10. A ball is denoted by $\bullet$ while an empty box by $\circ$.  }\label{basket}
	
\end{figure}
In \S\ref{bbs-sec} we describe  alternative equivalent descriptions of the dynamics. We denote $T\eta$ the configuration obtained after the carrier has visited all boxes in $\eta$, and $T^t\eta$ the configuration obtained after the iteration of this procedure $t$ times, for positive integer~$t$. The dynamics can be defined for suitable configurations with infinitely many balls  satisfying that ``there are more empty boxes than occupied boxes'' and conserves the set of configurations with density of balls less than $1/2$  \cite{FNRW}; see details in \S\ref{bbs-sec}. 

The main motivation of \cite{TS} was to identify  objects conserved by the dynamics that they called \emph{basic sequences}, later  called \emph{solitons} by \cite{LLP}; we follow this nomenclature. The Box-Ball system has been proposed as a discrete model with the same behavior of the Korteweg-de Vries equation \cite{TTMS}, an integrable partial differential equation having solitonic behavior.

Given a ball configuration $\eta$ a $k$-soliton consists of $k$ occupied boxes denoted $\tth_1,\dots,\tth_k\in\Z$ and $k$ empty boxes denoted $\ttt_1,\dots,\ttt_k\in\Z$. Takahashi and Satsuma showed that if $\eta$ has a finite number of occupied boxes, then all occupied box belongs to some soliton and proposed an algorithm to identify solitons.
We explain their algorithm in detail in \S\ref{secTSd}; for the moment we give the simplest example of $k$-soliton. Let $\eta$ have only $k$ balls occupying $k$ successive boxes, then the (only) $k$-soliton of $\eta$ consists on the $k$ successive occupied boxes  $\tth_1,\dots,\tth_k$ and the $k$ successive empty boxes $\ttt_1,\dots,\ttt_k\in\Z$ given by $\ttt_j=\tth_j+k$.

For the $\eta$ just described, the configuration $T\eta$ has balls in boxes  $\ttt_1,\dots,\ttt_k$ and no balls in the other boxes,  implying that $\eta'=T\eta$ has a $k$-soliton consisting of occupied boxes $\tth_1',\dots,\tth_k'$ with $\tth_j'=\ttt_j$ and empty boxes $\ttt_j'= \ttt_j+k$. Hence, in one step, an isolated $k$-soliton preserves its shape and moves $k$ steps forward. Iterating the evolution $t$ times, we conclude that not being other balls in the system, a $k$-soliton moves $kt$ boxes forward, that is, it travels at speed $k$. Since for different $k$'s the solitons have different speeds, they ``collide'', that is the order and the positions of $\tth_i$ and $\ttt_i$ of each soliton change. In fact, solitons can be identified for configurations with finite number of balls \cite{TS} and for suitable configurations $\eta$ with infinitely many balls \cite{FNRW} and moreover solitons are conserved by the dynamics \cite{TS} \cite{FNRW}; we explain this in detail in \S\ref{sec3}. Given a suitable configuration $\eta$, a $k$-soliton consists always on $k$ occupied boxes and $k$ empty boxes, but they are not necessarily consecutive and the empty boxes of the soliton may precede the occupied ones. In any case, different solitons occupy disjoint sets of boxes. The trajectory of each soliton can be identified along time \cite{FNRW}. When the distribution of the initial ball configuration is translation invariant and invariant for the dynamics, the asymptotic soliton speeds satisfy a system of linear equations \cite{FNRW} which is a feature of several other integrable systems \cite{cbs19}.

A ball configuration can be mapped to a walk indexed by boxes that jumps one unit up at occupied boxes and one unit down at empty boxes. If the configuration has density less than $\frac12$, then the walk has down \emph{records} and finite \emph{excursions} consisting on the pieces of configuration between two consecutive records. A ball configuration can be codified as a set of infinite vectors, based on the concept of \emph{slots}~\cite{FNRW}. Given a ball configuration with identified solitons, there are boxes called $k$-slots satisfying that any $k$-soliton is strictly in between two successive $k$-slots; in this case we say that the $k$-soliton is \emph{attached} to the left $k$-slot. To be more precise, the set of $k$-slots of $\eta$ consists on the records and the boxes belonging to any bigger soliton of the form $\tth_j$ or $\ttt_j$ for any $j>k$. Taking a ball configuration with a record at the origin, the $k$-slots are enumerated and then for each integer $i$,  the $i$-th coordinate of the $k$-component of the configuration is the number of $k$-solitons attached to the $i$-th $k$-slot. The obtained components can be composed again to recover the initial configuration $\eta$. A large part of the paper is dedicated to a complete explanation of these constructions.

It is useful to perform the decomposition in each excursion to obtain what \cite{fg18} call \emph{slot diagram}, a combinatorial object that encodes the structure of the excursion. We discuss also the relationship between the soliton decomposition of an excursion and other combinatorial objects as the excursion tree \cite{MR0290475,MR0370775,MR942038,evans, LG}, Catalan numbers  \cite{evans} and Dyck and Motzkin paths \cite{evans, LLP}.

A notable property proven by \cite{FNRW} is that the $k$-component of the configuration $T\eta$ is a shift of the $k$-component of $\eta$, the amount shifted depending on the $m$-components for $m>k$. As a consequence,  \cite{FNRW} prove that measures with independent and translation invariant soliton components are invariant for the dynamics.
In \cite{fg18} a special class of these measures is studied in detail. The papers \cite{CKST, croydon2019invariant} show families of invariant measures for the BBS based on reversible Markov chains on $\{0,1\}$.

The Box-Ball system is strictly related to several remarkable combinatorial constructions (see for example \cite{IKT, KTZ, LLP, MIT, YYT, YT}); we illustrate some of them. Sometimes, instead of giving formal proofs and detailed descriptions we adopt a more informal point of view trying to illustrate the different constructions through explicative examples.


\smallskip

The paper is organized as follows.

In \S\ref{sec2} we fix the notation and review several different equivalent definitions of the dynamics. We start considering the simple case of a finite number of balls. Then, following \cite{FNRW} we introduce the walk representation and give a definition of the dynamics in the general case for configurations of balls whose walk representation can be cut into infinitely many finite excursions.

In \S\ref{sec3} we discuss some conserved quantities of the dynamics, the identification of the solitons, a codification of the conserved quantities in terms of Young diagrams and define the \emph{slot diagrams} from \cite{fg18}.

In \S\ref{sec4} we recall the construction of the excursion tree and propose a new soliton decomposition of the excursion based on the tree. We introduce a branch decomposition of the tree and conclude that its slot diagram coincides with the one discussed in \S\ref{sec3}. The contents of this section are new.

In \S\ref{sec5}  we review results from \cite{fg18} related with the distribution of the soliton decomposition of excursions. In particular, the soliton decomposition of a simple random walk consists on independent double-infinite vectors of iid geometric random variables. We discuss also an application to branching processes.

\section{Preliminaries and notation}
\label{sec2}

\subsection{Box-Ball System}\label{bbs-sec}

The \emph{Box-Ball System} (BBS) \cite{TS} is a discrete-time cellular automaton. We start considering a finite number of balls evolving on the infinite lattice $\mathbb Z$. The elements of $\mathbb Z$ are called boxes.
A configuration of balls is codified by $\eta\in \{0,1\}^\mathbb Z$, that is, by a doubly infinite sequence of $1's$ and $0's$, corresponding respectively to the boxes occupied by balls and the empty boxes. Pictorially a ball will be denoted by $\bullet$ while an empty box by $\circ$.

There are several equivalent ways of defining the evolution. We denote by $T:\{0,1\}^{\mathbb Z}\to\{0,1\}^{\mathbb Z}$ the operator defining the evolution in one single step. This means that the configuration $\eta$ evolves in a single step into the configuration $T\eta$. In the following definitions we consider configurations having only a finite number of $1's$.

The equivalence among all the definitions is simple. The different definitions are however related to different classic combinatorial constructions and illustrate
the evolution from different perspectives.

\smallskip

\underline{First definition} We define the dynamics through a pairing between the balls  and some empty boxes.
Consider a ball configuration $\eta$ containing only a finite number of balls.
The evolution is defined iteratively.
At the first step we consider the balls that have an empty box in the nearest
neighbor lattice site to the right, that is, local configurations
of the type $\bullet\circ$ and we pair the two boxes drawing a line. Remove all the pairs created and continue following the same rule with the configuration obtained after the deletion of the
paired boxes. This procedure will stop after a finite number of iterations because there are only a finite number of balls.
See Fig.~\ref{cerchi}, where we assumed that there are no balls outside the window and the lines connect balls with the corresponding paired empty boxes.
The evolved configuration of balls, denoted $T\eta$
is obtained by
transporting every ball along the lines to the corresponding paired empty box.
Note that the lines pairing balls and empty boxes
can be drawn without intersections in the upper half plane.

\begin{figure}[h!]
	
	\centering
	
	\includegraphics[trim={0 10mm 0 19mm},clip]{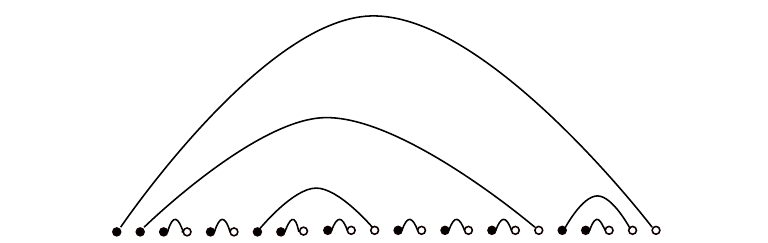}
	
	\caption{A finite configuration of balls with the corresponding pairing lines of the first definition.  }\label{cerchi}
	
\end{figure}
\smallskip

\underline{Second definition} \cite{TS}: This is the original definition of the model. Consider
an empty carrier that starts to the left of the leftmost ball and visit the boxes one after another moving from left to right. The carrier can transport an arbitrary large
number of balls.
When visiting box $i$, the carrier picks the ball if $\eta(i)=1$ and the number of balls transported by the carrier augments therefore by one and site $i$ is updated to be empty: $T\eta(i)=0$.  If instead $\eta(i)=0$ and the
carrier contains at least one ball then he deposits one ball in the box getting $T\eta(i)=1$. After visiting a finite number of boxes the carrier will be always empty and will not change any more the configuration, see Figure \ref{basket}. The final configuration $T\eta$ is the same as the one obtained by the previous construction.

\smallskip

\underline{Third definition}: Dyck words (after Walther von Dyck). Substitute any ball with an open parenthesis and any empty box with a closed one. The sequence of Fig.~1 becomes for example
$$
((()()(()())()()())(()))
$$
and outside this window there are only closed $)$ parenthesis. According to the usual algebraic rules we can pair any open parenthesis to the corresponding closed one. Recalling that open parenthesis correspond to balls, we move each ball from the
position of the open parenthesis to the position of the corresponding closed one.

\smallskip

\underline{Forth definition}: As a first step we duplicate each ball. After this operation on each occupied box there will be exactly 2 balls, one is the original one while the second is the clone. We select an arbitrary occupied box and move the cloned ball to the first empty box to the right. Then we select again arbitrarily another box containing two balls and do the same. We continue according to an arbitrary order up to when there are no more boxes containing more than one ball. At this point we remove the original balls and keep just the cloned ones. The configuration of balls that we obtain does not depend on the arbitrary order that we followed and coincides with $T\eta$.

\smallskip

\underline{Fifth definition}: Start from the leftmost ball and move it to the nearest empty box to its right. Then do the same with the second leftmost ball (according to the original order). Proceed in this way up to move once all the balls. This is a particular case of the fourth definition. It correspond to move the balls according to the order given by the initial position of the balls.

\smallskip

Our viewpoint will be to consider all the balls indistinguishable and from this perspective all the above definitions are equivalent. If we are instead interested in the motion of a tagged ball then  we can have different evolutions according to the different  definitions given above.

The construction can be naturally generalized to a class of configurations with infinitely many balls or to configuration of balls on a ring. This can be done under suitable assumptions on the configuration $\eta$ \cite{CKST,FNRW}. We will discuss briefly this issue following the approach of \cite{FNRW}, but
to do this we need some notation and definitions.

\subsection{Walk representation and excursions}\label{ex}

A function $\xi:\Z\to\Z$ satisfying $|\xi(i)-\xi(i-1)|=1$ is called \emph{walk}.
We map a ball configuration $\eta$ to a  walk $\xi=W\eta$
defined up to a global additive constant by
\begin{align}
\label{x11}
\xi(i)-\xi(i-1)=2\eta(i)-1
\end{align}
The constant is fixed for example by choosing $\xi(0)=0$. Essentially
the map between ball configurations and walks is fixed by the correspondence $\bullet\longleftrightarrow \diagup$ and $\circ\longleftrightarrow\diagdown$, where
$\bullet$ represents a ball, $\circ$ an empty box and $\diagup,\diagdown$ pieces of walk to be glued together continuously.  The map $W$ is invertible (when the additive constant is fixed) and the configuration of balls $\eta=W^{-1}\xi$ can be recovered using \eqref{x11}. We remark that there are several walks that are projected to the same configuration of balls and all of them differ by a global additive constant. This means that $W$ is a bijection only if the arbitrary additive constant is fixed and this will be always done in such a way that $\xi(0)=0$.

We call $i\in \mathbb Z$ a (minimum) \emph{record} for the walk $\xi$ if $\xi(i)<\xi(i')$ for any $i'<i$. The hitting time of $-j$ for the walk $\xi$ is a record denoted $r(j,\xi)$. We call \emph{excursion} of a walk the piece of trajectory between two successive records. A pictorial perspective on the decomposition of the walk into records and disjoint excursions is the following. Think the walk as a physical profile and imagine the sun is at the sunshine on the left so that the light is coming horizontally from the left.
The parts of the profile that are enlightened correspond
to the records while the disjoint parts in the shadow are the different excursions.

We call a \emph{finite walk} a finite trajectory of a random walk. More precisely a finite walk $\xi=(\xi(i))_{i\in \mathbb [0,k]}$, $k\in \mathbb N$, is an element of $\mathbb Z^{[0,k]}$ such that $|\xi(i)-\xi(i-1)|=1$. Again we always fix $\xi(0)=0$ and like before there is a bijection $W$ between finite walks and
finite configurations of balls, i.e. elements $\eta\in \{0,1\}^k$ for some $k\in \mathbb N$.
We use the same notation $\xi$ for finite and infinite walks and $\eta$ for finite and infinite configurations of balls. It will be clear from the context when the walk/configuration is finite or infinite.

We introduce the set $\mathcal E$ of \emph{finite soft excursions between records 0 and 1}.
An element $\vep\in \mathcal E$ is a finite walk that starts and ends at
zero, it is always non-negative and it has length $2n(\vep)$. More precisely
$\vep=\Big( \vep(0),\dots, \vep(2n(\vep))\Big)$ with the constraints $|\vep(i)-\vep(i-1)|=1$, $\vep(i)\ge 0$ and
$\vep(0)=\vep(2n(\vep))=0$.
The empty excursion $\emptyset$ is also an element of $\cE$ with $n(\emptyset)=0$.
We call $\mathcal E_n$ the set of soft finite excursions of length $2n$ so that
$\mathcal E=\cup_{n=0}^{+\infty}\mathcal E_n$.

Using the same correspondence as before between walks and configuration of balls
we can associate a finite configuration of balls $\big(\eta(1),\dots ,\eta(2n(\vep))\big)=W^{-1}\vep$ to the finite excursion $\vep$. If $\eta=W^{-1}\vep$, then we have
$\sum_{i=1}^{2n(\vep)}(2\eta(i)-1)=0$ but obviously not all configuration of balls satisfying this constraint generates a soft excursion by the transformation $W$.
It is well known \cite{C} that the number of excursions of length $2n$ is given by
\begin{equation}\label{catalan}
|\mathcal E_n|=\frac{1}{n+1}\binom{2n}{n}\,;
\end{equation}
the right hand side is the Catalan number $C_n$.

We denote by $\cE^o\subset \mathcal E$ the set of strict excursions. An element $\vep \in \mathcal E^o$ is an excursion that satisfies the strict inequality $\vep(i)>0$ when $i\neq 0,2n(\vep)$. Likewise we call $\mathcal E^o_n$ the strict excursions of length $2n$.

There is a simple bijection between $\mathcal E_n$ and $\mathcal E^o_{n+1}$. This is obtained by considering an element $\vep \in \mathcal E_n$ and adding a
$\diagup$ at the beginning and a $\diagdown$ at the end. The result is an element of $\mathcal E^o_{n+1}$. The converse map is obtained removing a $\diagup$ at the beginning and a $\diagdown$ at the end of an element of $\mathcal E^o_{n+1}$ obtaining an element of $\mathcal E_n$. This can be easily shown to be a bijection. In particular we deduce by \eqref{catalan} that $|\mathcal E^o_n|=\frac{1}{n}\binom{2(n-1)}{n-1}$.

\paragraph{Concatenating excursions}
Given a finite soft excursion $\vep$ we call $\tilde\vep$ the finite
 walk $\left(\tilde\vep(i)\right)_{i=0}^{2n(\vep)+1}$ such that $\tilde\vep(i)=\vep(i)$ when $0\leq i\leq 2n(\vep)$ and $\tilde\vep(2n(\vep)+1)=-1$.
 This corresponds essentially to add a $\diagdown$ at the end of the soft excursion.
 Given two such finite walks $\tilde\vep_0$ and
 $\tilde \vep_1$ we introduce their concatenation $\tilde\vep_0\star\tilde\vep_1$. This is a finite walk such that $\left[\tilde\vep_0\star\tilde\vep_1\right](i)=\tilde\vep_0(i)$ when $0\leq i\leq 2n(\vep_0)+1$
 and $\left[\tilde\vep_0\star\tilde\vep_1\right](i)=\tilde\vep_1(i-2n(\vep_0)-1)-1$ if
 $2n(\vep_0)+1< i\leq 2(n(\vep_0)+n(\vep_1))+2$. Essentially this operation
 corresponds to glue the graphs of the walks one after the other
 continuously. Iterating this operation we can define similarly
 also the concatenation of a finite number of finite walks $\tilde\vep_0\star\tilde\vep_2\star \dots \star \tilde\vep_k$. Likewise we
 consider an infinite walk $\left(\tilde\vep_i\right)_{i\in \mathbb Z}^{\star}$ obtained by a doubly infinite concatenation of finite walks. Informally this is obtained concatenating continuously the graphs as before with the condition that
 $\left(\tilde\vep_i\right)_{i\in \mathbb Z}^{\star}(j)=\tilde\vep_0(j)$ for
 $0\leq j\leq 2n(\vep_0)+1$.
\begin{figure}[h!]
	
	\centering
	
	\includegraphics{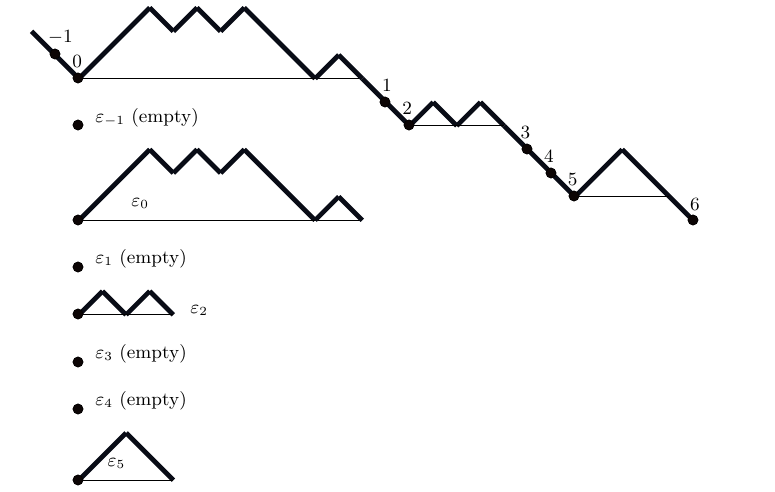}
	
	\caption{Up: walk representation of  ball configuration of next Fig.~\ref{cerchi-q}, with records represented by a black dot and labeled from $-1$ to 6. Down: excursion decomposition of the walk. }\label{forest-soliton}
	
      \end{figure}
Formally the walk $\xi$ is defined in terms of a family of excursions $(\vep_j)_{j\in\Z}$ as follows. First fix the position of the records of the walk $\xi$ iteratively by
\begin{align}
r(0,\xi)&:= 0,\nn\\
r(k+1,\xi)-r(k,\xi)&:=   2n(\vep_k)+1 \quad \hbox{for }k\in \Z\nn
\end{align}
so that the number of boxes between records $k$ and $k+1$ is the size of excursion $k$. Now complete the definition by inserting excursion $k$ between those records:
\begin{align}
\xi(r(k,\xi)+i) = -k+\vep_k(i), \quad i\in\{0,\dots,2n(\vep_k)\},\quad k\in\Z.
\end{align}
The resulting walk $\xi$ attains the level $-k$ for the first time at position $r(k,\xi)$. In particular, $\xi$ has infinite many records, one for each element of $\mathbb Z$. When this happens we say shortly that the walk has all the records.
Clearly a similar  concatenation procedure can be performed for any collection of finite walks and not just for excursions. We do not give the straightforward details.

Conversely, if we have a walk $\xi$ with all the records and such that record 0 is at 0 and record $k$ is at $r(k,\xi)$, then for each $k\in\Z$
we can define the excursion $\vep_k=\vep_k[\xi]$ by
\begin{align}
\vep_k(i) = \xi(r(k,\xi) +i)-(-k),\quad i\in\{0,1,\dots,r(k+1,\xi)-r(k,\xi)-1\}.
\end{align}
then we have that $\left(\tilde\vep_i\right)_{i\in \mathbb Z}^{\star}$ coincides with the original walk $\xi$.

We proved therefore that an infinite walk is obtained by an infinite concatenation of finite soft excursions separated by a $\diagdown$ if and only if it has all the records.

The set of configurations with density $a$ is defined by
\begin{equation}\label{banca}
\cX_a:=\left\{\eta\in\{0,1\}^\mathbb Z: \lim_{n\to \pm\infty}\frac{\sum_{j=0}^n\eta(j)}{n+1}=a\right\}\,, \qquad a\in [0,1]\,,
\end{equation}
and call $\cX:=\cup_{a<1/2}\cX_a$, the set of configurations with some density below $\frac12$. Consider $\eta\in \cX$ and let $\xi=W\eta$.
Since the walk $\xi$ is a nearest neighbor random walk with negative drift, it can assume any given value $k\in \mathbb Z$ only a finite number of times and therefore the walk will have all the records and hence we have that any element of $W\cX$ can be seen as a concatenation of infinitely many finite excursions.

The converse statement is however in general not true. It is possible to construct walks concatenating finite excursions that belong to $\cX_{1/2}$ or also such that the limits involved in the definition \eqref{banca} do not exist.

An example for the first case is a concatenation $\left(\tilde\vep_i\right)^*_{i\in \mathbb Z}$ where the walk $\tilde\vep_i$
is obtained adding an $\diagdown$ to the excursion $\vep_i$ that has length
$2^{|i|+1}$ and is composed by an alternating sequence of $\diagup$ and $\diagdown$.

An example for which the limits do not exist is when the excursion $\vep _i$ is formed by  a sequence of $2^{|i|}$ pieces of the type $\diagup$ followed by the same numbers of $\diagdown$.




\subsection{BBS with infinitely many balls and on the ring}
\label{s23}

We can now generalize the definition of the dynamics to infinite configurations of balls. This can be done in a natural way under suitable assumptions. In particular the dynamics can be defined for configuration of balls  whose corresponding walk has all the records.
We have already shown that in this case the walk is a suitable horizontal translation of the concatenation $\left(\tilde\vep_i\right)^*_{i\in \mathbb Z}$  of infinite many finite excursions with a $\diagdown$ appended at the end.
We define also the action of the operator $T$ on  configurations of balls on a ring with $N$ sites containing $k\leq \frac{N}{2}$ balls. In both cases the basic idea is that we can define the action of the evolution operator $T$ on each single excursion of the decomposition of an associated walk.

We discuss this issue using the first definition of the dynamics. Similar arguments can be given also for the other definitions. The basic fact is that, when the walk
of an infinite configuration of balls has all the records then in the pairing procedure all the lines joining balls and empty boxes can be constructed locally. More precisely drawing a vertical line going through a record $r(k,\eta)\in \mathbb Z$ we have that there are no lines
of the construction that cross this vertical line. All the balls belonging to an excursion are paired to empty boxes belonging to the same excursion.

Therefore, if the walk representation $W\eta$ of an infinite configuration of balls $\eta$ is the concatenation of infinitely many finite excursions separated by records, then $T\eta$ can be naturally defined, using the first definition. More precisely the operative definition of $T$ is the following. Consider an excursion of the walk and consider the balls that are in the corresponding lattice sites. Erase all the other balls of the configuration. In this way we obtain a configuration with a finite number of balls and we can apply the original first definition. All the balls will be paired with boxes belonging to lattice sites of the excursion. We do this for all the excursions of the walk. In this way we obtain the configuration $T\eta$. There are no overlaps since all the constructions stay inside the disjoint excursions of $W\eta$.

The example of Fig.~\ref{cerchi} corresponds to a configuration of balls having one single non empty excursion. The example of Fig.~\ref{cerchi-q} corresponds instead to a configuration having 3 non empty excursions that are surrounded by rectangles. The lines constructed for any excursion
are naturally divided into blocks. These blocks correspond exactly to the natural
subdivision of any excursion into the concatenation of strict excursions. Each block has a maximal line surrounding all the others and there are no other lines surrounding the maximal ones. This means that balls and empty boxes corresponding to a strict excursion are paired among themselves and therefore the evolution of the balls of each strict excursion is determined independently of what happens outside. In Fig.~\ref{cerchi-q} the leftmost excursion is the concatenation of 2 strict excursions and correspondingly there are 2 maximal lines inside the rectangle. The same happens to the central excursion while the rightmost excursion has only one maximal line so that the excursion is strict.

We observe  (see \cite{CKST} for more details) that there are configurations $\eta$ such that $T\eta$ is well defined but  $T(T\eta)$ is not. For simplicity we use a configuration $\eta$ that is built up by the concatenation of infinite strict excursions, but we could as well start from a configuration with infinitely many excursions separated by records. The configuration $\eta$ that we consider is $\eta=(\vep_i)_{i\in \mathbb Z}^*$ (the definition of this concatenation is straightforwardly similar to the one given for excursions with a $\diagdown$ appended at the end) where the excursions $\vep_i$ with $i\geq 0$ are all obtained by ball configurations of the form
$\bullet\circ$ while the excursion $\vep_i$ with $i<0$ is of the form $$\overbrace{\bullet\bullet\dots\bullet }^{|i|\ \textrm{balls}}\overbrace{\circ\circ\dots \circ}^{|i|\ \textrm{empty boxes}}$$
By our previous arguments it is possible to implement the transformation $T$ since we can operate separately on each strict excursion. As the reader can easily see it is not possible to define $T(T\eta)$. This is because the configuration $T\eta$ has no records and hence cannot be divided into finite disjoint excursions.

It is important to see that if $\eta\in \cX$ then we can define $T^k\eta$ for any $k$. This is because it can be easily shown that $T$ maps elements of $\cX$ to elements of $\cX$. The example just illustrated does not indeed belong to $\cX$.

A similar discussion can be done using the other definitions in \S\ref{bbs-sec}. Let us consider the second definition. In the case of a finite number of balls we imagine the carrier starting empty just on the left of the first ball. In the case of infinitely many balls we can consider however the carrier starting empty in correspondence of a record and moving to the right. The carrier is performing a transformation on the configuration of balls corresponding to the first excursion that he meets. After this he will reach a new record box and correspondingly he will be again empty. Then the carrier can proceed afresh to the second excursion and so on. This means that equivalently the transformation $T$ can be performed by infinitely many carriers, one for each finite excursion. They start empty to the left of the excursion and end empty at the right of the excursion. The evolved configuration $T\eta$ can therefore  be computed locally restricting to each single excursion.
\begin{figure}[h!]
	
	\centering
	
	\includegraphics[trim={0 10mm 0 35mm},clip]{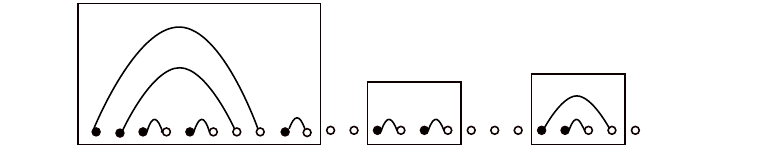}
	
	\caption{An example of a pairing between balls and empty boxes
		containing several maximal lines. The 3 different excursions are surrounded
		by rectangles.  Maximal lines correspond to strict excursions. }\label{cerchi-q}
	
\end{figure}

The definition of the dynamics on a ring can be done simply associating to each configuration on the ring an infinite periodic configuration on $\mathbb Z$. When  the number of balls is strictly less than $N/2$ we have that the corresponding walk has all the records and we can perform the construction as discussed above on each excursion independently. The evolved configuration is again periodic and can be interpreted as a ball configuration on the ring.
This fact does not hold in the case that the number of balls is exactly $N/2$ since in this case the infinite associated walk will have no records. However the dynamics in this case consists simply in flipping  the value of each box. Empty boxes becomes full while full boxes becomes empty.

\smallskip

\underline{Sixth definition}: Using the walk representation of a configuration of balls it is possible to give another equivalent definition of the Box-Ball dynamics. Since we know that the evolution operator $T$ acts independently on each excursion let us consider a configuration of balls $\eta$ on $\mathbb Z$ having just a finite number of balls  and such that the corresponding walk has one single excursion. The updating rule of the evolution $T$ corresponds in flipping the graph of the excursion like in Fig.~\ref{T-6-3}. When there are more than one single excursion the same symmetry operation has to be done on each single excursion. The configuration $T\eta$ is recovered applying $W^{-1}$ to the new walk obtained. This dynamics was already proposed by Le Gall \cite{MR942038}.

\begin{figure}[h!]
	
	\centering
	
	\includegraphics[trim={0 13mm 0 10mm},clip] {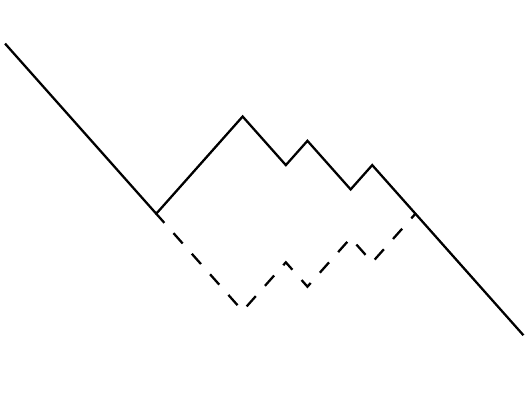}
	
	\caption{The walk of a configuration with finite many balls and having one single excursion. The action of $T$ consists in flipping the graph of the part of the walk corresponding to the excursion. After the action of $T$ the graph of the excursion has to be substituted by the dashed part.}\label{T-6-3}
	
\end{figure}

\underline{Seventh definition}: Here we write in formulas the construction done in the above definition. These formulas apply directly to infinite configurations of balls having all the records. The first simple and general formula that summarize the evolution is
\begin{equation}
T\eta(x)=\left\{
\begin{array}{ll}
0 & \textrm{if}\ x\ \textrm{is\ a \ record\ of}\ W\eta\\
1-\eta(x) & \textrm{otherwise}\,.
\end{array}
\right.
\end{equation}
The second formula is the following. For a walk $\xi$ having all the records, the
curve $\min_{y\leq x}\xi(y)$ is well defined. The operator $T$ essentially reflects the walk $\xi$ with respect to this curve. We have
\begin{equation}
T\xi(x)=\left[\min_{y\leq x}\xi(y)\right]-\left[\xi(x)-\min_{y\leq x}\xi(y)\right]\,,
\end{equation}
where we denote by $T\xi$ the walk corresponding to $T\eta$ when $\xi=W\eta$, i.e.
$T\xi:=WT\eta$.

\section{Conserved quantities and solitons}
\label{sec3}

In this section we discuss how to identify the solitons that are traveling through
the system. We obtain different combinatorial structures and discuss the relationship among them. Solitons are  conserved quantities of the system.

\subsection{Runs}
\label{sdfr}

Given a configuration of balls $\eta$, a \emph{zero-run} is a maximal integer interval of empty boxes and a \emph{one-run} is a maximal integer interval of boxes occupied by balls; we call run the union of these to sets. The runs of $\eta$ form a partition of the lattice $\mathbb Z$. In statistical mechanics a run is usually called a cluster.

More precisely the finite interval $[x,y]\subseteq \mathbb Z$ is a run
if $\eta(z_1)=\eta(z_2)$ for any $z_1,z_2\in [x,y]$ and moreover
$\eta(x-1)\neq \eta(x)$ and $\eta(y+1)\neq \eta(y)$. A run can be
a zero run or a one run depending if the boxes are
respectively empty or occupied. A run can be also semi-infinite or infinite. We can have therefore runs of the form $(-\infty,x]$
or $[x,+\infty)$ or even $(-\infty,+\infty)$. In the first case we have $\eta(z_1)=\eta(z_2)$ for any $z_1,z_2\leq x$ and $\eta(x+1)\neq \eta(x)$, in the second case we have $\eta(z_1)=\eta(z_2)$ for any $z_1,z_2\geq x$ and $\eta(x-1)\neq \eta(x)$ while in the last case we have $\eta(z_1)=\eta(z_2)$ for any $z_1,z_2$.

Given the run $[x,y]$ we call $|x-y|$ its size. The size of the semi-infinite or infinite runs is $+\infty$.

Any configuration of balls generates a partition of $\mathbb Z$
into disjoint runs alternating between zero-runs and one-runs.
A configuration containing a finite nonzero number of balls induces a finite collection of runs being the leftmost and the rightmost semi-infinite zero-runs. The configuration with no balls has only a double infinite zero-run. In this situation it is possible to implement the following algorithm.

\subsection{Takahashi-Satsuma soliton decomposition}
\label{secTSd}
The following is a small variation of the original algorithm.
\begin{itemize}
	\item[0)] Start with a ball configuration $\eta$ with a finite number of balls.
	\item[1)] If there is just one single infinite zero-run then stop, otherwise go to the next step.
	
	\item[2)]  Search for the leftmost among the smallest runs; assume that the run size is $k$. The $k$ boxes belonging to the run  and the first $k$ boxes belonging to the nearest neighbor run to its right (whose size is necessarily not smaller than $k$) identify a soliton $\gamma$, call $\tth_1(\gamma)<\dots<\tth_k(\gamma)$ the positions of the occupied boxes of $\gamma$ and $\ttt_1(\gamma) <\dots <\ttt_k(\gamma)$ the positions of the empty boxes of $\gamma$.
	
	\item[3)] Ignore the boxes of the identified solitons, update the runs gluing together the remaining boxes and go to step 1.
	
\end{itemize}
Since there is a finite number of balls, the algorithm stops after a finite number of iterations and identifies a finite number of solitons. This algorithm is called TS decomposition. For instance, the TS decomposition of the excursion of Fig.~\ref{excursion} is given in Fig.\ref{TS-decomposition}.

The soliton decomposition of a ball configuration with infinitely many balls and infinitely many records is done performing the above algorithm on each single excursion.
\begin{figure}[h!]

\centering

\includegraphics{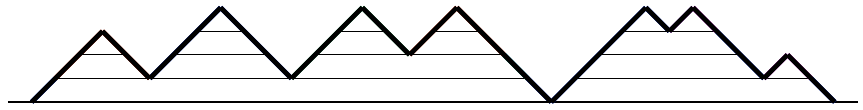}

\caption{An excursion.}\label{excursion}

\end{figure}
\begin{figure}[h!]

\centering

\includegraphics{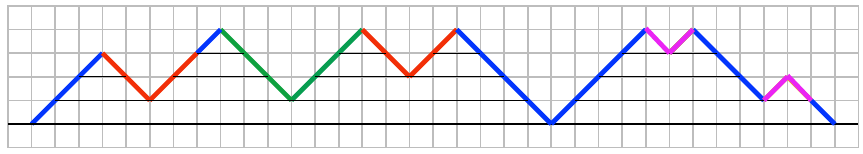}

\caption{TS-decomposition of the excursion in Fig.~\ref{excursion}. First iteration: there are 3 smallest runs of size 1. Take the leftmost one and its next box (which turns out to be also a smallest run) to identify a 1-soliton; pretend that the identified boxes are not present and update the runs accordingly. Second iteration: there is now one smallest run of size 1, identify the second 1-soliton. In the 3rd and 4th iterations the smallest runs are of size 2; In the 5th iteration the 2 smallest runs are of size 3; for the 6th and 7th iteration the smallest runs are of size 4. We have identified 2 1-solitons (violet), 2 2-solitons (red), 1 3-soliton (green) and 2 4-solitons (blue).
}\label{TS-decomposition}

\end{figure}%
Each soliton $\gamma$ is composed by two disjoint sets of the same cardinality, the set of occupied boxes $\tth(\gamma)$, called the \emph{head} and the set of empty boxes $\ttt(\gamma)$ called the \emph{tail}. They satisfy $\gamma=\ttt(\gamma)\dot\cup \tth(\gamma)$. If $|\ttt(\gamma)|=|\tth(\gamma)|=k$ we call $\gamma$ a $k$-soliton and write $\ttt(\gamma)=(\ttt_1(\gamma),\dots ,\ttt_k(\gamma))$ with $\ttt_i(\gamma)<\ttt_{i+1}(\gamma)$ and $\tth(\gamma)=(\tth_1(\gamma),\dots ,\tth_k(\gamma))$ with $\tth_i(\gamma)<\tth_{i+1}(\gamma)$.
Note that either $\tth_i(\gamma)<\ttt_j(\gamma)$ for any $i,j$ or $\tth_i(\gamma)>\ttt_j(\gamma)$ for any $i,j$ and that the walk has the same height at $\tth_i(\gamma)$ and $\ttt_i(\gamma)$: $\xi(\tth_i(\gamma))=\xi(\ttt_i(\gamma))$, for this reason we say that the head and tail of a soliton are \emph{paired}.

We have the following key definition.
\begin{definition}\label{defslot}
We say that a box $i$ is a \emph{$k$-slot} if either $i$ is a record or $i$ belongs to $\{\ttt_\ell(\gamma),\tth_\ell(\gamma)\}$ for some $\ell> k$ for some $m$-soliton $\gamma$ with $m>k$.
\end{definition}
The set of $k$-slots contains the set of $m$-slots for all $m>k$. We illustrate in Fig.~\ref{TS-slot} the slots induced by the soliton decomposition of the excursion in Fig.~\ref{TS-decomposition},
\begin{figure}[h!]

\centering

\includegraphics{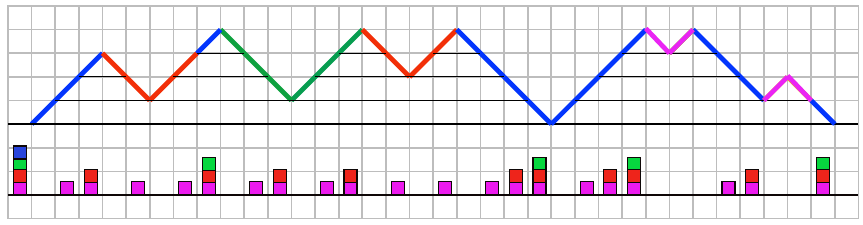}

\caption{Slots induced by the TS-decomposition of the excursion in Fig.~\ref{TS-decomposition}. Violet, red, green and blue squares are respectively 1-, 2-, 3- and 4-slots, respectively. In the extreme left, the record preceding the excursion is $k$-slot for all $k\ge1$.}\label{TS-slot}

\end{figure}
obtaining one 4-slot (at the record), 5 3-slots, 11 2-slots and 21 1-slots. According to definition \ref{defslot}, the number of $k$-slots is $s_k:=1+\sum_{\ell>k} 2(\ell-k)n_k$, where $n_k$ is the number of $k$-solitons in the excursion. The $1$ in the above formula corresponds to the record on the left that is a slot of any order.

For each $k$ we enumerate the $k$-slots in the excursion starting with 0 for the $k$-slot in the record preceding the excursion, We say that a $k$-soliton $\gamma$ is \emph{attached} to the $k$-slot number $i$ if the boxes occupied by $\gamma$ are contained in the segment with extremes the $i$th and $(i+1)$th $k$-slots in the excursion. We define
\begin{align}
  \label{xki}
  x_k(i) := \text{\#\{$k$-solitons attached to $k$-slot number $i$\}}.
\end{align}
\subsection{Slot diagrams}
We define a combinatorial family of objects called \emph{slot diagrams} that according to \cite{FNRW} is in bijection with $\mathcal E$, see also \S\ref{sec4} below.

A generic slot diagram is denoted by $x=\left(x_{k}: 1\leq k\leq M\right)$, where
$M=M(x)$ is a non negative integer,
$$
x_{k}=\left(x_{k}(0), \dots ,x_{k}(s_k-1)\right)\in \mathbb N^{s_k}
$$
and $s_k$ is a non-negative integer. We say that $x_k(j)$ is the number of $k$-solitons attached to the $k$-slot number $j$.
We denote $n_{k}:=\sum_{i=0}^{s_k-1}x_k(i)$, the number of $k$-solitons in $x$. A precise definition is the following.

We say that $x$ is a \emph{slot diagram} if
\begin{itemize}

\item There exists a non negative integer number $M=M(x)$ such that
$s_M=1$ and $x_m(0)=0$  for $m>M$. Hence, we can ignore soliton sizes above $M$ and denote $x=\left(x_{k}: 1\leq k\leq M\right)$.

\item For any $k$, the  number of $k$ slots $s_k$ is determined by $(x_\ell:\ell>k)$ via the formula
\begin{equation}\label{sk}
s_k=s_k(x)=1+2\sum_{i=k+1}^{M}(i-k)n_{i}\,.
\end{equation}
\end{itemize}
Consider now the soliton decomposition of an excursion and the corresponding vectors defined by \eqref{xki} and define $M= \min\{k\ge 0: x_{k'}(0) = 0$ for all $k'>k\}$. Then, the family of vectors $(x_k: k\le M)$ forms a slot diagram; if $M=0$ the slot diagram is \emph{empty} and corresponds to an empty excursion. The slot diagram of the excursion in Fig.\ref{TS-slot} is as follows. We have $M=4$ and
\begin{align}
x_4&=(2), \quad s_4=1;\nonumber\\
x_3&= (0,1,0,0,0), \quad s_3=5;\nonumber\\
x_2&=(0,1,0,0,1,0,0,0,0,0,0),\quad s_2=11;\nonumber\\
x_1&= (0,0,0,0,0,0,0,0,0,0,0,0,0,0,0,0,0,1,0,1,0),\quad s_4=21.\nonumber
\end{align}
For example the vector $x_3$ has just $x_3(1)=1\neq 0$ since there is just one 3-soliton and its support (the boxes corresponding to the green part of the walk in Fig.~\ref{TS-slot}) is contained between the 3-slots number 1 and number 2.
Recall that the $k$-slot located at the record to the left of the excursions is numbered 0 for all $k$ and then the 3-slot number 1 is the second green box from the left in Fig.~\ref{TS-slot}.

\subsection{Head-Tail soliton decomposition} We propose another decomposition, called HT soliton decomposition.
\begin{itemize}
\item[0)] Start with a ball configuration $\eta$ with a single excursion.
\item[1)] If there is just one single (infinite) run then stop, otherwise go to the next step.
\item[2)]  Search for the leftmost among the smallest runs. If the run contains 1's, then pair the boxes belonging to the run with the first boxes (with zeroes) belonging to the nearest neighbor run to its right. If the run contains 0's, then pair the boxes with the nearest boxes (with ones) to the left of the run. The set of paired boxes and their contents identifies a soliton $\gamma$.
\item[3)] Ignore the boxes of the identified solitons, update the runs gluing together the remaining boxes and go to step 1.
\end{itemize}
The HT soliton decomposition of the excursion in Fig.\ref{excursion} is given in Fig.\ref{14-heta}. The name of the decomposition comes from the fact that the head of each soliton is to the left of its tail in all cases. We will denote soliton$^{\diamond}$ those  solitons identified  by the HT decomposition. We will see that this decomposition arises naturally in terms of a tree associated to the excursion.
\begin{figure}[h!]
	
	\centering
	
	\includegraphics{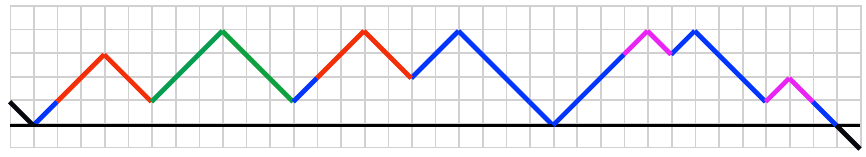}
	
	\caption{The HT soliton decomposition of the excursion in Fig.\ref{excursion}.}\label{14-heta}
	
\end{figure}

We say that a box $i$ is a $k$-slot$^{\diamond}$ if either $i$ is a record or $i\in\{\tth_\ell(\gamma^{\diamond}),\ttt_{m-\ell+1}(\gamma^{\diamond})\}$ for some $\ell\in\{1,m-k\}$ for some $m$-soliton$^{\diamond}$ $\gamma^{\diamond}$ for some $m>k$; for example, if $\gamma^{\diamond}$ is a 4-soliton$^{\diamond}$, $\tth_1(\gamma^{\diamond})$ and $\ttt_4(\gamma^{\diamond})$ are 3-slots$^{\diamond}$. See the upper part of Fig.\ref{16-ht+ts}. Observe that, as before, the set of $k$-slots$^{\diamond}$ is contained in the set of $\ell$-slots$^{\diamond}$ for any $\ell<k$.
 \begin{figure}[h!]
	\centering
	\includegraphics[width= \textwidth] {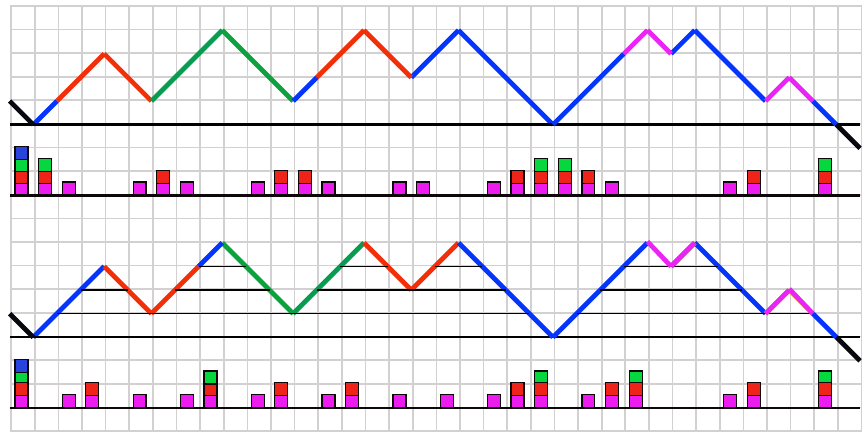}
	\caption{Comparing the HT decomposition (above) with the Takahashi-Satsuma decomposition (below). Colored squares indicate that the corresponding box is a $k$-slot$^{\diamond}$ (or a $k$-slot) with 1 = violet, 2 = red, 3 = green and 4 = blue. The number of $k$-slots belonging to an $\ell$-soliton is $2(\ell-k)$ in both cases, but the localizations inside the solitons are different. The leftmost box corresponds to the record the excursion is associated with. The slot diagrams of both decompositions coincide; for instance there is a 3-soliton attached to 3-slot number 1 in both pictures.}\label{16-ht+ts}
      \end{figure}
      As before we say that a $k$-soliton$^{\diamond}$ is attached to $k$-slot$^{\diamond}$ number $i$ if the boxes of the soliton$^{\diamond}$ are strictly between $k$-slots$^{\diamond}$ $i$ and $i+1$.
      If we enumerate the $k$-slots$^{\diamond}$ of the excursion starting with 0 for the $k$-slot$^{\diamond}$ at record 0, we can again define $x_k^{\diamond}(i):=$ number of $k$-solitons$^{\diamond}$ attached to $k$-slot$^{\diamond}$ number $i$. This produces a slot diagram $x^{\diamond}$ associated to the excursion. We denote $x^{\diamond}[\vep]$ the slot diagram produced by the HT soliton decomposition of the excursion $\vep$.

      See Fig.\ref{16-ht+ts} for a comparison of the slots  induced by the HT soliton decomposition and the TS soliton decomposition.

      The next result says that the slot diagrams produced by both decompositions are identical. Observe that a slot diagram gives information about the number of solitons and about their combinatorial arrangement so that codifies completely the corresponding excursion.


\begin{theorem}
  \label{ts=td}
  The slot diagram of the Head-Tail decomposition to an excursion $\vep\in\cE$ coincides with the slot diagram of the Takahashi-Satsuma  decomposition of $\vep$. That is, $x[\vep]=x^{\diamond}[\vep]$.
\end{theorem}

\begin{proof} Let $\vep$ be an excursion and denote $x[\vep]$ and $x^{\diamond}[\vep]$ the TS and HT slot diagrams of $\vep$, respectively; let $m$ and $m^{\diamond}$ be the maximal soliton size in each representation. Let $\tts_k(i)$ and $\ttts_k(i)$ be the position of the $i$-th $k$-slot in the TS and HT decompositions of $\vep$, respectively.

  Assume $\vep$ has neither $\ell$-solitons nor $\ell$-solitons$^{\diamond}$ for all $\ell\le k$. Then, $\tts_k(0)=\ttts_k(0)=0$ and for $0<i<s_k$, we will show that
\begin{align}
  \label{20s}
    \tts_k(i) =
    \begin{cases}
      \ttts_k(i) + k, &\text{if $\tts_k(i)$ belongs to the head of a soliton$^{\diamond}$};\\
   \ttts_k(i),&\text{if  $\tts_k(i)$ belongs to the tail of a soliton$^{\diamond}$}.
    \end{cases}
\end{align}
which implies the theorem.
We prove \eqref{20s} by induction. If $\vep$ has only $m$-solitons, then \eqref{20s} holds for any $k<m$ by definition. Assume \eqref{20s} holds if  $\vep$ is an excursion with no $\ell$-solitons for $\ell\le k$. Now attach a $k$-soliton$^{\diamond}$ $\gamma^{\diamond}$ to $\ttts_k(i)$ and a $k$-soliton $\gamma$ to $\tts_k(i)$.

We have 2 cases:

(1) $s_k(i)$ is the record or belongs to the tail of a $m$-soliton $\alpha$ with $m$ bigger than $k$. In this case also $s_k^{\diamond}(i)$ belongs to the record or to the tail of a $m$-soliton$^{\diamond}$ $\alpha^{\diamond}$ and $\gamma$ is attached to the same place as $\gamma^{\diamond}$, hence it does not affect the distances between $\ell$-slots and $\ell$-slots$^{\diamond}$ in the excursion --indeed, they coincide in the record and in the tail of $\alpha$ and $\alpha^{\diamond}$-- for $\ell\le k$. On the other hand, the $\ell$-slots carried by $\gamma$ and the $\ell$-slots$^{\diamond}$ carried by $\gamma^{\diamond}$ satisfy \eqref{20s}.

(2) $s_k(i)$ is in the head of $\alpha$. In this case necessarily $s^{\diamond}_k(i)$ is in the head of $\alpha^{\diamond}$ by inductive hypothesis and $s_k(i)= s^{\diamond}_k(i)+k$. We consider 2 cases now:

(2a) $k$-slots. The attachments of $\gamma$ to $s_k(i)$ and  $\gamma^{\diamond}$ to $s_k^{\diamond}(i)$ does not change the distance between $k$-slots and $k$-slots$^{\diamond}$ because either $s_k(j)< s_k(i)$ and $s^{\diamond}_k(j)< s^{\diamond}_k(i)$ and in this case the insertions do not change their positions or otherwise both slots are translated by $2k$, the number of boxes occupied by the $k$-solitons. We conclude that \eqref{20s} is satisfied by $k$-slots and $k$-slots$^{\diamond}$ after the attachments.
 \begin{figure}[h!]
	\centering
	\includegraphics[width= .8\textwidth] {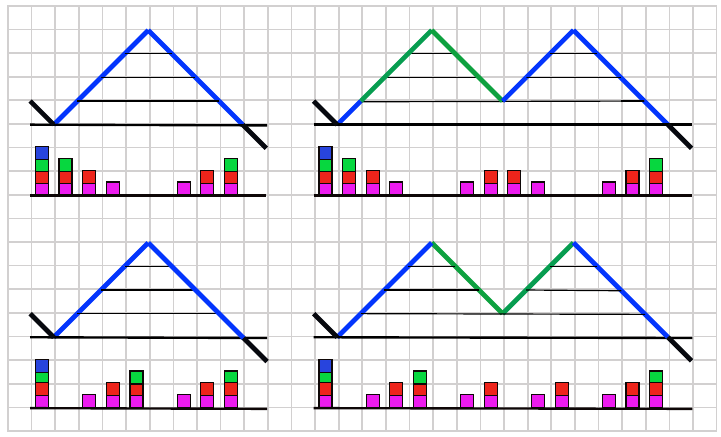}
	\caption{Above: Attach a $3$-soliton$^{\diamond}$ $\gamma^{\diamond}$ in $3$-slot$^{\diamond}$ number 1. Below: attach a 3-soliton $\gamma$ in 3-slot number 1. The excursions after the attachments coincide and the slot$^{\diamond}$ diagram after the attachments coincide with the slot diagram and is given by  $m=4$, $x_4=(1)$, $x_3=(0,1,0)$, $x_2$ is a vector with 7 zeroes and $x_1$ a vector with 11 zeroes. The soliton decompositions of the excursion obtained after the attachments satisfy \eqref{20s}.}\label{18-b}
      \end{figure}

      (2b) $\ell$-slots for $\ell<k$. Take an $\ell<k$ and an $\ell$-slot $s_\ell(j)$ in the head of $\alpha$. If $s^{\diamond}_\ell(j)<s^{\diamond}_k(i)$ and $s_\ell(j)<s_k(i)$, neither will be displaced, so \eqref{20s} is satisfied for $\ell$-slots to the left of $s_k(i)$. On the other hand, if $s^{\diamond}_\ell(j)> s^{\diamond}_k(i)$, then  $s^{\diamond}_\ell(j)$ keeps its place after the attachment of $\gamma^{\diamond}$ and $s_\ell(j)$ is to the left of the attachment, hence they satisfy \eqref{20s} after the attachments (this is the case of the 4th violet $1$-slot and $1$-slot$^{\diamond}$).

We have proved that if the slot and slot$^{\diamond}$ diagrams of an excursion with no $\ell$-solitons for $\ell\le k$ coincide, then they coincide after attaching $k$-solitons and $k$-solitons$^{\diamond}$.
\end{proof}

\subsection{Attaching solitons}
\label{appsoli}

In the previous subsection we discussed the decomposition of a configuration into
elementary solitons/solitons$^{\diamond}$ and how to codify each single excursion using a slot
diagram that takes care of the combinatorial arrangement of the solitons/solitons$^{\diamond}$
into the available slots/slots$^{\diamond}$. In this Section we discuss the reverse construction.
Given a slot diagram we illustrate how to construct the corresponding excursion.
The procedure is particularly simple and natural in the case of the HT decomposition.

\begin{figure}[h!]
	
	\centering
	
	\includegraphics{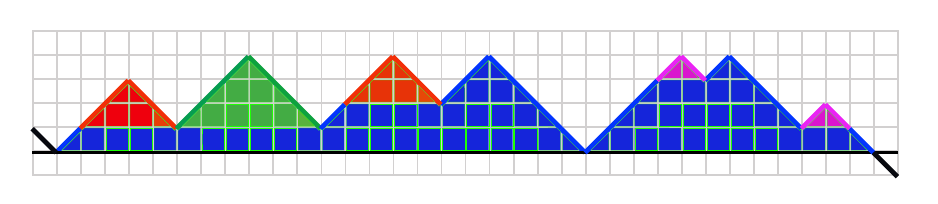}
	
	\caption{The epigraph of an excursion divided into different colored regions obtained by drawing horizontal lines from the leftmost point on the graph of the excursion associated to a given soliton$^{\diamond}$ to the rightmost in the
		HT decomposition. }\label{14-heta-dav}
	
\end{figure}
The basic idea is illustrated in Fig.~\ref{14-heta-dav}, which was obtained from Figure
\ref{14-heta} by drawing horizontal lines from the leftmost point in the graph of the excursion associated to the head to the rightmost point associated to the tail of each soliton$^{\diamond}$. These lines cut the epigraph of the excursion into disjoint regions that we color with the corresponding color
of the boundary. We imagine each colored region as a physical two dimensional object glued recursively to generate the interface. Indeed we will show that the excursion can be obtained as the final boundary of a region obtained adding with a tetris-like construction one after the other upside oriented triangles having elastic diagonal sides
\begin{figure}[h!]
	
	\centering
	
	\includegraphics[trim={0 25mm 0 25mm},clip]{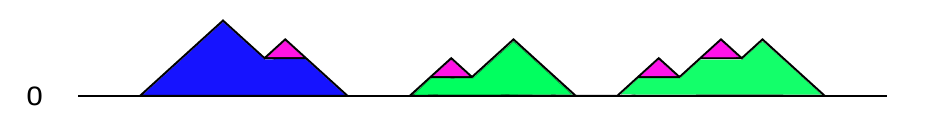}
	
	\caption{The walk of a ball configuration where the down-steps $\diagdown$ associated to records have been substituted by horizontal lines $\frac{\ \ }{\ \ }$ at height zero. The region below the graph of each excursion has been colored like in Fig.~\ref{14-heta-dav}}\label{marimonti}
	
\end{figure}

It is convenient to represent the walk associated to a ball configuration in $\cX$ as follows. We transform each down oriented step $\diagdown$ associated to a record into an horizontal line $\frac{\ \ }{\ \  }$ at height $0$. The parts of the walk associated to the excursions are vertically shifted to level 0, remaining concatenated one after the other by an horizontal line of length equal to the number of records separating the excursions in the walk. The walk is therefore represented by infinitely many pieces of horizontal lines at the zero level (the \emph{sea} level) separated by infinitely many finite excursions (mountain profiles). This is the construction associated to the \emph{Harris walk} (see for example \cite{LLP}). See Fig.~\ref{marimonti} for an  example with three excursions where we implemented also the same coloring of Fig.~\ref{14-heta-dav}.

\begin{figure}[h!]
	
	\centering
	
	\includegraphics[width=.6 \textwidth]{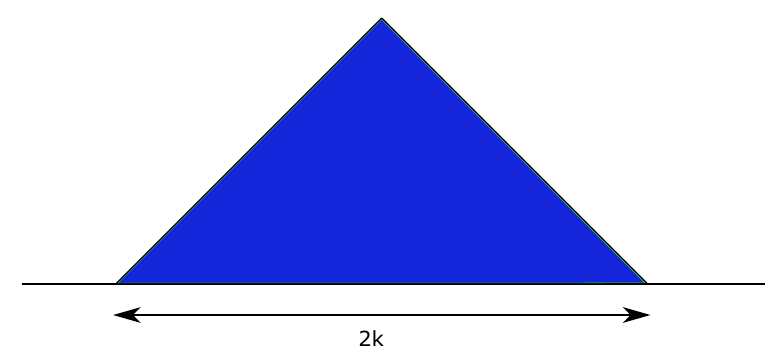}
	
	\caption{An isolated $k$-soliton$^{\diamond}$ represented as a right angle triangle
		with horizontal hypotenuse of size $2k$, up oriented and having the other sides
		of equal length. The hypotenuse is rigid while the other sides are soft and deformable. }\label{solidso}
	
            \end{figure}
            We discuss how to generate
one single excursion from a slot diagram using the HT
decomposition.
We represent an isolated $k$-soliton$^{\diamond}$ as a
right-angle isosceles triangle having hypotenuse of size $2k$. The triangle is oriented in such a way that the hypotenuse is horizontal and the triangle is upside oriented, see Fig.~\ref{solidso}.

The basic mechanism of attaching solitons$^{\diamond}$ is illustrated in Fig.~\ref{4-1-solitoni}.
In the first up left drawing we represent a 4 soliton$^{\diamond}$ as an upper oriented triangle
and draw below it the corresponding slots$^{\diamond}$. The leftmost slot$^{\diamond}$ corresponds to a record located just on the left of the excursion. Colors are like before: violet=1, red=2, green=3, blue=4. In the drawing number $i$ with $i=0,\dots ,6$ we attach
one 1-soliton$^{\diamond}$ to the 1-slot$^{\diamond}$ number $i$. This corresponds to attach a triangle with
horizontal hypotenuse of size 2 in correspondence of the position of the corresponding slot$^{\diamond}$.
\begin{figure}[h!]
	
	\centering
	
	\includegraphics[trim={0 6mm 0 7mm},clip]{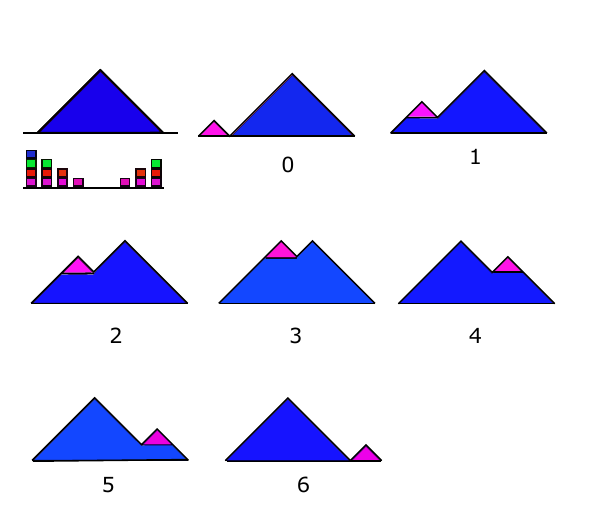}
	
	\caption{An isolated 4-soliton$^{\diamond}$ (blue, up left diagram) with the corresponding $k$-slot$^{\diamond}$ for $k=1,2,3,4$ (colors as in the previous Section). The record to the left of the excursion is the unique $4$-slot$^{\diamond}$. We attach one 1-soliton$^{\diamond}$ (violet) in all the possible ways. In the drawing number $i$ the 1-soliton$^{\diamond}$ is attached to the 1-slot$^{\diamond}$ number $i$.}\label{4-1-solitoni}	
\end{figure}
The Figure is exhaustive and represents all the possible
ways of attaching the 1-soliton$^{\diamond}$.  The precise rules and the change of the positions of the slot$^{\diamond}$ during the attaching procedure to generate an excursion, are illustrated using as an example the following slot diagram
\begin{align}\label{scottex}
k \to &\; x_k\nonumber\\
4 \to &\; (1)\nonumber\\
3 \to &\; (0,0,0)\\
2 \to &\;  (0,1,0,1,0)\nonumber\\
1 \to &\;  (0,0,0,1,0,0,0,0,0,0,0)\nonumber
\end{align}

We construct now the excursion that corresponds to this slot diagram. We do this  using the HT decomposition since it is simpler but the TS decomposition gives as a result the same excursion.
First we observe that the maximal soliton$^{\diamond}$ size in \eqref{scottex} is $4$ and there is just one maximal soliton$^{\diamond}$.
\begin{figure}[h!]
	
	\centering
	
	\includegraphics{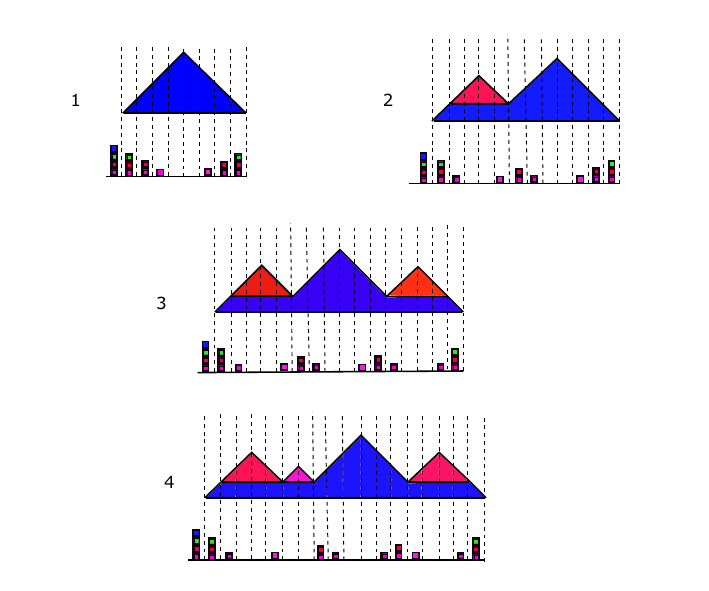}
	
	\caption{The growing of the excursion codified by the slot diagram \eqref{scottex}, adding solitons$^{\diamond}$ one after the other from the biggest to the smallest. Adding a new soliton$^{\diamond}$ corresponds to add an up oriented triangle with horizontal hypotenuse in correspondence of the slot$^{\diamond}$ specified by the diagram. The diagonal sides of the already presented triangles are soft and deform in order to glue perfectly the geometric figures.}\label{slor-ex}
	
\end{figure}
We start therefore with drawing 1 of Fig.~\ref{slor-ex}
where we have a blue 4-soliton$^{\diamond}$ represented  by a upside oriented triangle. Below it we represent also the $\ell$-slots$^{\diamond}$ for $\ell<4$; the leftmost $\ell$-slot$^{\diamond}$ is always located in the record just on the left of the excursion. Since there are no 3-solitons$^{\diamond}$ we do not have to add
green triangles having hypotenuse of size 6. We proceed therefore attaching 2-solitons$^{\diamond}$ represented as upside oriented triangles with hypotenuse of size  4. We have two of them and we have to attach to the 2-slot$^{\diamond}$ number 1 and 3. We label as $\ell$-slot$^{\diamond}$ number zero the one associated to the record and number the other ones increasingly
from left to right. There are 5 2-slot$^{\diamond}$ in the drawing 1 of Fig.~\ref{slor-ex} (that are the piles of colored squares containing a red one). We start attaching  the 2-soliton$^{\diamond}$ to the 2-slot$^{\diamond}$ number 1. This means that the left corner of the red triangle has to be attached to the boundary of the colored region in correspondence to
the intersection of the boundary with the dashed line just on the right of 2-slot$^{\diamond}$
number one. Since the bottom edge of the triangles is rigid the blue diagonal side deforms in order to have a perfect gluing. This is illustrated in the drawing number 2 of Fig.~\ref{slor-ex}. Note that the slots$^{\diamond}$ in correspondence with the shifted diagonal sides of the blue triangle are shifted accordingly. There are moreover
new 1-slot$^{\diamond}$ created in correspondence with some red diagonal sides. The same gluing
procedure is done with a second red triangle in correspondence of the 2-slot$^{\diamond}$ number 3, and this is shown in the drawing number 3 of Fig.~\ref{slor-ex}. Note that we do this two gluing operations one after the other to illustrated better the rules but they can be done simultaneously or in the reversed order, the final result is the same. This is because attaching a $k$-soliton we generate just new j-slot$^{\diamond}$ with $j<k$.
Finally we have to attach a 1-soliton$^{\diamond}$ that is a violet triangle in the 1-slot$^{\diamond}$
number 3 and this is shown in the final drawing 4 of Fig.~\ref{slor-ex}.

\subsection{Conserved quantities}

We discuss a way to identify conserved quantities using the first definition of the dynamics in \S\ref{bbs-sec} applied to a finite excursion. Recall that the basic step consists on pairing all neighboring boxes of type 10 by drawing a line from the 1 to the 0 and then remove the paired boxes to iterate, see Fig.\ref{cerchi}. Call $r_i$ the number of lines drawn in the $i$-th iteration of the construction. We have $r_1\geq r_2\geq\dots \geq r_M$, where $M$ is the number of iterations necessary to pair all the balls. In the example of Fig.~\ref{cerchi} we have $M=4$ and $r_1=8, r_2=2, r_3=r_4=1$.

\begin{proposition}[Yoshihara, Yura, Tokihiro \cite{YYT}]
  The numbers $r_i$ are invariant for the dynamics. That is,
  \begin{align}
    \label{ri=ri}
    r_i(\eta)=r_i(T\eta)\quad \text{ for any }i.
  \end{align}
\end{proposition}
\begin{figure}[h!]
	
	\centering
	
	\includegraphics[trim={0 15mm 0 2mm},clip]{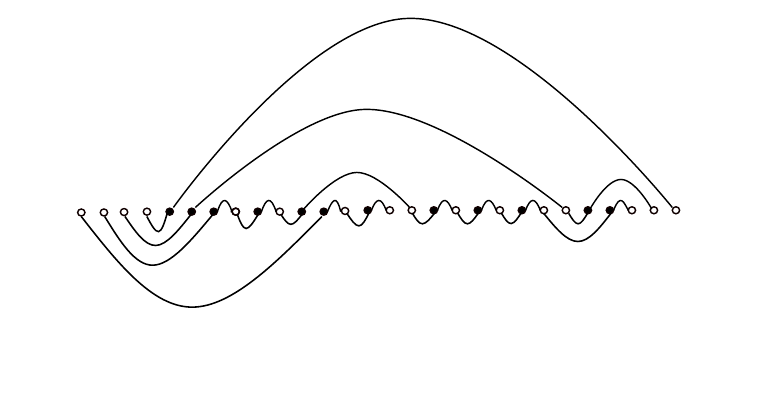}
	
	\caption{The pairing construction for a dynamics evolving to the right (lines above) and the pairing construction for the same configuration of balls but evolving to the left (lines below).}\label{cerchi-sotto}
	
\end{figure}

\begin{proof}
  We present a simplified version of the argument given by \cite{YYT}. The basic property that we use is the reversibility of the dynamics. Introduce the evolution $T^*$ that is defined exactly as the original dynamics apart the fact that balls move to the left instead of to the right. The reversibility of the dynamics is encoded by the relation $T^*T\eta=\eta$. This fact follows from the definition: looking at Fig.~\ref{cerchi} the configuration $T\eta$ is obtained just coloring black the white boxes and white the black ones. The evolution  $T^*$ is obtained pairing balls with empty boxes to the left. The lines associated to  $T^*$ for the configuration $T\eta$
are exactly the same as those already drawn. The only  difference is that the balls are now transported from right to left along these lines. Denote  $r^*_i$ the number of lines drawn at iteration number $i$ for the evolution $T^*$. Since the lines used are the same we have
\begin{equation}\label{euno}
r_i(\eta)=r^*_i(T\eta)\,, \qquad \forall i\,.
\end{equation}
Now evolve the original configuration $\eta$ according to $T^*$. 
In Fig.~\ref{cerchi-sotto} we draw above the lines corresponding to the evolution $T$ and below those corresponding to $T^*$. 
We want now to show that
\begin{equation}\label{edue}
  r_i(\eta)=r^*_i(\eta)\,,\qquad \forall i\,.
\end{equation}
Recall that a run is a sequence of consecutive empty or full boxes. In the configuration $\eta$ of our example there are two infinite empty runs and then alternated respectively 8 and 7 full and empty finite runs.

The first step is to show that $r_1(\eta)=r^*_1(\eta)$. This is simple because these numbers coincide with the number of full runs in the configuration $\eta$. The second step of the algorithm consists on erasing the rightmost ball of every occupied run and the leftmost empty box of every empty run for $T$, while the leftmost ball of every occupied run and the rightmost empty box of every empty run are erased for $T^*$. Observe that $r_2(\eta)$ coincides with the number of full runs in a configuration obtained removing the balls and the empty boxes paired in the first step. This configuration is obtained from $\eta$ decreasing by one the size of every finite run. If in $\eta$ there are some runs of size 1 then they disappear. The same happens for computing $r^*_2(\eta)$. Since we are just interested  on the sizes of the alternating sequences of empty and full runs, erasing on the left or on the right is irrelevant. We deduce $r_2(\eta)=r^*_2(\eta)$ since
both coincide with the number of finite occupied runs of two configurations having the same sequence of sizes of the runs.
Iterating this argument we deduce \eqref{edue}. Now,
using \eqref{euno} and \eqref{edue} we deduce \eqref{ri=ri}.
\end{proof}

\subsection{Young diagrams} We discuss now a generalization of the conservation property \eqref{ri=ri} to the case of infinite configurations and the relation with the conservation of the solitons.
Since the numbers $r_i$ are monotone, it is natural to represent them using a Young diagram, \cite{2018arXiv180808074K}.
A Young diagram is a diagram of left-justified rows of boxes where any row is not longer than the row on top of it.
We can fix for example the number $r_i$ representing the length of the row number $i$ from the top. The number of
iterations $M$ corresponds to the number of rows. The Young diagram associated to the example in Fig.~\ref{cerchi} is therefore
\begin{equation}\label{primoy}
\yng(8,2,1,1)\,.
\end{equation}
This diagram can be naturally codified by the numbers $r_i$, representing the sizes of the rows, as
$(8,2,1,1)$.
Another way of codifying a Young diagram is by the sizes of the columns. This gives another Young diagram that is called the \emph{conjugate} diagram and it is obtained by reflecting the diagram across the diagonal. The same diagram \eqref{primoy} can therefore be codified as $[4,2,1,1,1,1,1,1]$. Finally another equivalent codification can be given specifying the numbers $n_1, n_2, \dots , n_M$ of columns of length respectively
$1,2,\dots ,M$. For the Young diagram above we have for example $n_1=6, n_2=1, n_3=0, n_4=1$.  The numbers $r_i$ and $n_i$ give
alternative and equivalent coding of the diagram and are related by
\begin{equation}\label{uscita1}
r_i=\sum_{m=i}^{M}n_m\,, \qquad n_i=r_i-r_{i+1}\,
\end{equation}
where we set $r_{M+1}:=0$.

The number $n_i$ can be interpreted as the number of solitons of length $i$. Take for example the diagram \eqref{primoy} and cut it into vertical slices obtaining
\begin{equation}\label{fette}
\yng(1,1,1,1) \qquad \yng(1,1) \qquad \yng(1) \qquad \yng(1) \qquad \yng(1) \qquad \yng(1) \qquad\yng(1) \qquad
\yng(1)
\end{equation}
The original Young diagram can be reconstructed gluing together
the columns in decreasing order from left to right and justifying all of them to the top.
Each column of height $k$ in \eqref{fette} will represent a $k$-soliton on the dynamics. We are not giving a formal proof of this statement it can however easily be obtained by the construction in \S\ref{sipuo}. We will show indeed that the soliton decomposition can be naturally done using trees codifying excursions. In \S\ref{sipuo} we show how the trees can be constructed
using the lines of the first definition in \S\ref{bbs-sec} getting directly the relationship among the Young diagrams and the solitons.   According to this, the configuration $\eta$ having associated the Young diagram \eqref{primoy}  obtained
gluing again together the columns in \eqref{fette}, contains one 4-soliton one 2-soliton and 6 1-solitons.

The Young diagram contains only some information about the configuration of balls, i.e. the map that associates to $\eta$ its Young diagram is not invertible, and for example
there are several configurations of balls giving \eqref{primoy} as a result. The one in Fig.~\ref{cerchi} is just one of them. Essentially the Young diagram contains just the information concerning the numbers of solitons contained in the configuration but not the way in which they are combinatorially organized.

In the example discussed above we worked with a configuration of balls having one single non trivial finite excursion.
Consider now a finite configuration $\eta$ whose walk representation contains more than one excursion. Our argument on the conservation of the numbers $r_i$ proves that the global Young diagram associated to the whole configuration is invariant by the dynamics. Let us consider however separately the single excursions.
Recall that two different excursions are separated by empty boxes from which there are no lines exiting. For example in Fig.~\eqref{cerchi-q} there are 3 excursions that we surrounded by rectangles to clarify the different excursions.

We construct for each excursion separately the corresponding Young diagram.
For the example of Fig.~\ref{cerchi-q} the three Young diagrams are
\begin{equation}\label{ty}
\yng(3,1,1) \qquad \yng(2) \qquad \yng(1,1)
\end{equation}
By definition the global Young diagram that is preserved by the dynamics is the one having as length of the first row (the number $r_1$) the sum of the lengths of the first rows of the three diagrams, as length of the second row (the parameter $r_2$) the sum of the length of the second rows of all the Young diagrams and so on. This means that the global Young diagram is obtained suitably joining together the single Young diagrams. In particular the gluing procedure is the following.
We have to split the columns of each single diagram then put all the columns together and glue them together as explained before, i.e. arranging them in decreasing order from left to right and justifying all of them to the top. For example
the first Young diagram on the left in \eqref{ty} is split into
$$
\yng(1,1,1) \qquad \yng(1) \qquad \yng(1)
$$
For the second diagram in \eqref{ty} we have
two columns of size 1  $\yng(1)\ \yng(1)$ while for the third one we have one single
column of size 2 $\yng(1,1)$.

The global Young diagram for the example of Fig.~\ref{cerchi-q} is therefore
$$
\yng(6,2,1)
$$
The number $n_i$ of columns of length $i$ in the global diagram is obtained as the sum of the number of columns of size $i$ on the single diagrams. Also the numbers $r_i$ are obtained summing the corresponding row lengths on each single group (with the usual
convention that a Young diagram with $M$ rows has $r_j=0$ for $j>M$).

The shapes of the single diagrams in \eqref{ty} are not invariant by the dynamics. Even the number of such diagrams is not conserved since during evolution the number of excursions may change. It is instead the total number of columns of each given size to be conserved. More precisely given a configuration $\eta$ we can construct the Young diagrams for each excursions and then we can cut them into single columns. The configuration of balls
$T\eta$ will have different excursions with different Young diagrams but they will be obtained again combining differently into separated Young diagrams the same columns obtained for the configuration $\eta$. The Box-Ball dynamics  preserves the number of columns of size $k$ for each k. Indeed this is nothing else that a different identification of the traveling solitons again by the construction in \S\ref{sipuo}.

If $\eta$ is an infinite configuration with a walk having all the records, we can construct a Young diagram for each excursion. Cutting the diagrams along the columns we obtain the solitons contained in the excursion.

\emph{Slot diagrams and Young diagrams. } Since a slot diagram describes the number of solitons per slot, we can associate a Young diagram to a
slot diagram $x$ as follows:  $M(x)$ is the
number of rows and  $n_k$ is the number of columns of length $k$.
The diagram  is constructed gluing  $n_M$ columns of length $M$, then $n_{M-1}$ columns of length $M-1$ up to $n_{1}$ columns of length $1$. For example the Young diagram associated to the slot diagram \eqref{scottex} is given by
\begin{equation}\label{gloria}
\yng(4,3,1,1)\,.
\end{equation}

\section{Trees, excursions and slot diagrams}
\label{sec4}

In this section we provide an alternative decomposition of an excursion using a bijection between soft excursions and planar trees. The construction is a slight variant of the classical bijection of strict excursions and planar rooted trees, see \cite{MR942038,evans,LG}. There are several ways of codifying planar trees (see for example \cite{LG}). We will try to use a direct pictorial approach introducing less algebraic notation as possible.

\subsection {Tree representation of excursions}
\label{sected}
In this subsection we summarize classical results mapping finite trees to excursions, see for instance \S1.1 of the lecture notes of Le Gall \cite{LG}.  Start with the graph of a soft excursion as in Fig.~\ref{excursion-tree}. Draw horizontal lines corresponding to the integer values of the height. The region below the graph of the excursion is cut into disjoint components by the horizontal lines. Associate one node to each connected component. The root is the node corresponding to the bottom region.  The tree is obtained by drawing an edge between nodes whose associated components share a piece of a horizontal line. The construction is illustrated in Fig.~\ref{excursion-tree} where the root is drawn as a $\bullet$ while the other nodes as a  $\circ$.
\begin{figure}[h!]

\centering

\includegraphics[width=.8\textwidth]{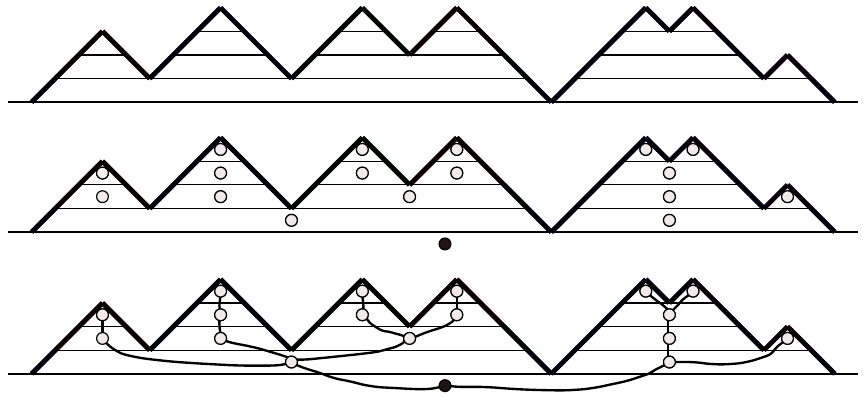}

\caption{Construction of a planar tree associated to the excursion of Fig.\ref{excursion}. The root is a black circle and the nodes are white circles. The root is associated to a record.}\label{excursion-tree}

\end{figure}
The tree that we obtain is rooted since there is a distinguished vertex and it is planar. This means that it is embedded on the plane where the graph
of the excursion is drawn. A consequence of this specific embedding is that every vertex different from the root has an edge incoming from below and all the other edges
are ordered from left to right going clockwise.

\smallskip
\begin{figure}[h!]
	
	\centering
	
	\includegraphics[width=.8\textwidth]{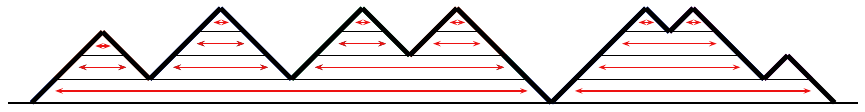}
	
	\caption{The pairing of opposite diagonal sides of an excursion. Each double arrow corresponds to a node of the tree different from the root.}\label{excursion-pairing}
	
      \end{figure}
In \cite{LG} the map from the planar tree to the excursion is given in terms of a Dyck path. The excursion gives the distance to the root of a vehicle that turns around the tree at speed one edge per unit of time.
The reverse bijection amount to glue the edges face to face below the
excursion (in order to recover the edges), as in Fig.\ref{excursion-pairing}.

\subsection{Trees and pairing algorithm}\label{sipuo}      The tree associated to an excursion can be constructed using the pairing definition of the dynamics  of Fig.~\ref{cerchi}. As before, draw dashed horizontal lines in correspondence of the integer heights that cut the epigraph of the excursion into disjoint regions. Pair the opposite diagonal faces of each region, connected by dashed double arrows in Fig.~\ref{excursion-pairing}. Since the left face is of type  $\diagup$ and the right one is of the type $\diagdown$, corresponding respectively to balls and empty boxes, we obtain exactly the pairing of the first definition of the dynamics. Indeed, the pairings of the first iteration of the first definition of the dynamics coincide exactly with the pairing of the two opposite diagonal sides near each local maxima. Then remove the paired objects and iterate to obtain a proof.

We construct the planar tree associating the root to the unbounded upper region of the upper half plane and one node to each pairing line. Nodes associated to maximal lines are linked to the root. Consider a node A associated to a maximal line. Node B associated to another line is connected to A if: 1) the line associated to B is surrounded by the line associated to A and 2) removing the maximal line associated to A the line associated to B becomes maximal. The tree is constructed after a finite iteration of this algorithm, see Fig.~\ref{neveu-red} where the planar tree is red and downside oriented.

\begin{figure}[h!]
	
	\centering
	
	\includegraphics[width=.4 \textwidth]{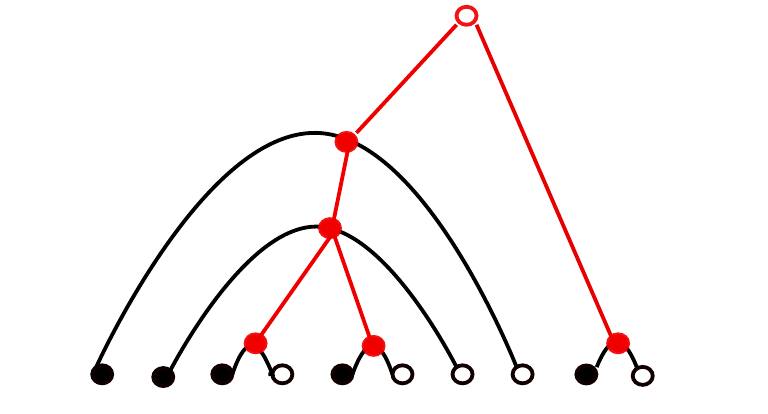}
	
	\caption{The construction of the planar tree associated to an excursion using the pairing between balls and empty boxes.}\label{neveu-red}
	
\end{figure}

\subsection{Branch identification of planar trees}\label{ftd}
We discuss a natural branch decomposition of a rooted planar tree that is in correspondence with the soliton decompositions previously discussed.

We give 3 equivalent algorithms to identify the \emph{branches} of a planar rooted tree.

\emph{Branch identification I}

Step 1. Let $A_1$ be the set of the leaves (nodes with only one neighbor). Associate a distinct color and  the \emph{generation number} 1 to each leaf. The root is black, a color not allowed for the other nodes.

Step $\ell$. Let $A_{\ell-1}$ be the set of numbered and colored nodes after $\ell-1$ steps. Let $\ttN_\ell$ be the set of nodes with all offsprings in $A_{\ell-1}$. To each $\ttn\in \ttN_\ell$ give the color of the rightmost neighbor among those with bigger generation number, say $g$, and give generation number $g+1$ to $\ttn$. Stop when all nodes are colored.
\begin{figure}[h!]

	\centering
	
	\includegraphics[width=.6 \textwidth] {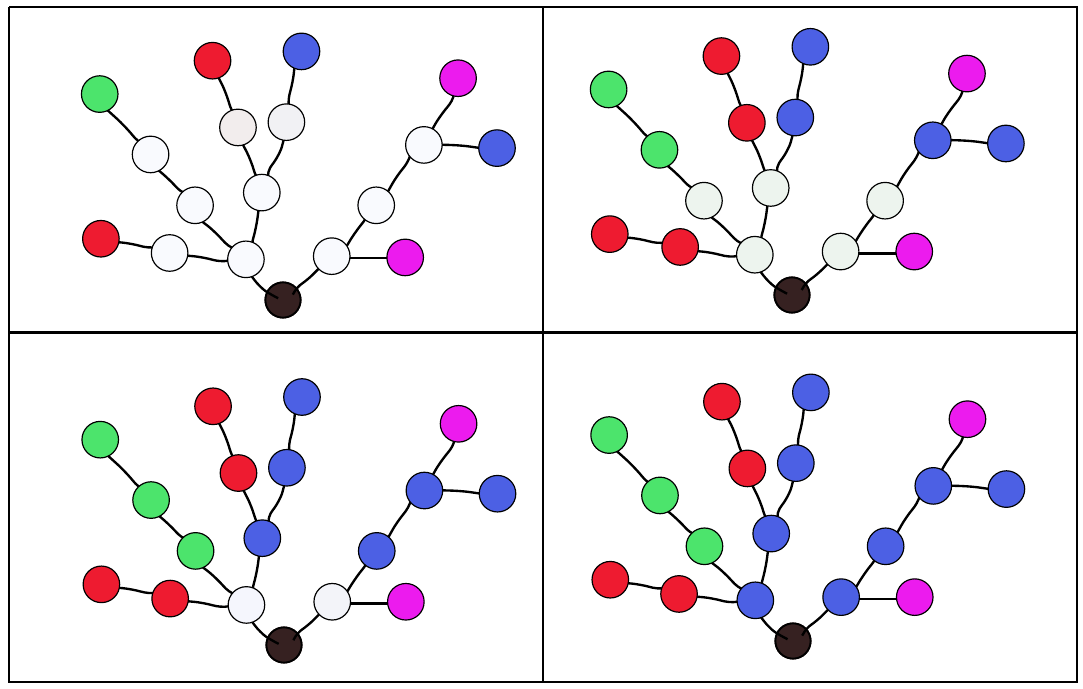}
	
	\caption{Branch identification I. }\label{14-uc}
      \end{figure}

      In Fig.~\ref{14-uc} give a distinct color to each leaf (we have for simplicity repeated colors in the picture). In each step to each not-yet-colored node with all offsprings already colored give the color of the rightmost maximal offspring. After coloring all nodes, identify the color of branches of the same size (knowing the result, we have started with those colors already identified).

A \emph{$k$-branch} is a one-dimensional path with $k$ nodes all of the same color and $k$ edges, one of which is incident to a node of a different color. In Fig.\ref{14-uc} we have colored the tree produced by the excursion in Fig.\ref{excursion-tree} and have identified 2 violet 1-branches, 2 red 2-branches, 1 green 3-branch and 2 blue 4-branches (for simplicity we used a simplified convention for color, see the caption for the explanation).

\emph{Branch identification II}

Step 0. Enumerate the colors. In our example we use violet for 1-branches, red for 2-branches, green for 3-branches and blue for 4-branches.

Step 1. Paint all leaves with color 1, violet.

Step $\ell$. Update those nodes with all offsprings entering into nodes already colored during steps 1 up to  $\ell-1$. Give color $\ell$ to updating nodes and change to color $\ell$ those nodes belonging to the rightmost offspring path of size $\ell$ starting from each updating node. See Fig.~\ref{14-dav}.
\begin{figure}[h!]
	
	\centering
	
	\includegraphics[width=.6 \textwidth] {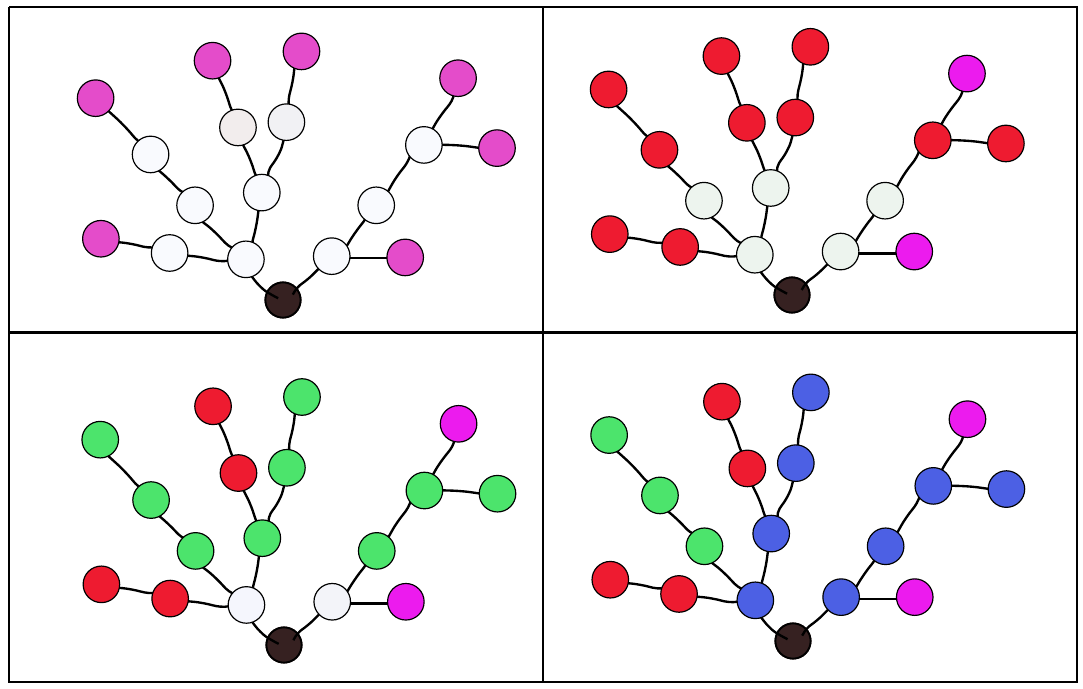}
	
	\caption{Branch identification II.}\label{14-dav}
	
      \end{figure}

      In Fig.~\ref{14-dav} we give color 1 (violet in this case) to each leaf. In step $2$ (a) give color $2$ (red) to all nodes having all offsprings already colored and (b) change to color $2$ each already colored node belonging to the rightmost offspring path with $2$ nodes starting at each updating node. In step 3 use color green and in step 4 use color blue. The final branch decomposition is the same as in Fig.\ref{14-uc}.
 \FloatBarrier 

\emph{Branch identification III}

Step 0. Orient the tree toward the root. Consider the oriented paths starting from the leaves of the tree.  Remove the root but not the edges incident to the root.

Step 1: Search for the maximal directed paths starting from the leaves. If two or more of them share at least one edge, select just the rightmost path among those. Observe that the last edge is incident only to one node. A selected path with $k$ nodes is named $k$-branch. Remove the selected branches.

Step 2. If all paths have been removed, then stop. Otherwise go to step 1.

The tree is oriented just to define the procedure. The branches selected and removed constitute the branch decomposition of the tree. In Figure \ref{15-dav} we apply this procedure to the same example of the previous procedures. The result is the same.
\begin{figure}[h!]
	
	\centering
	
	\includegraphics[width=.6 \textwidth]{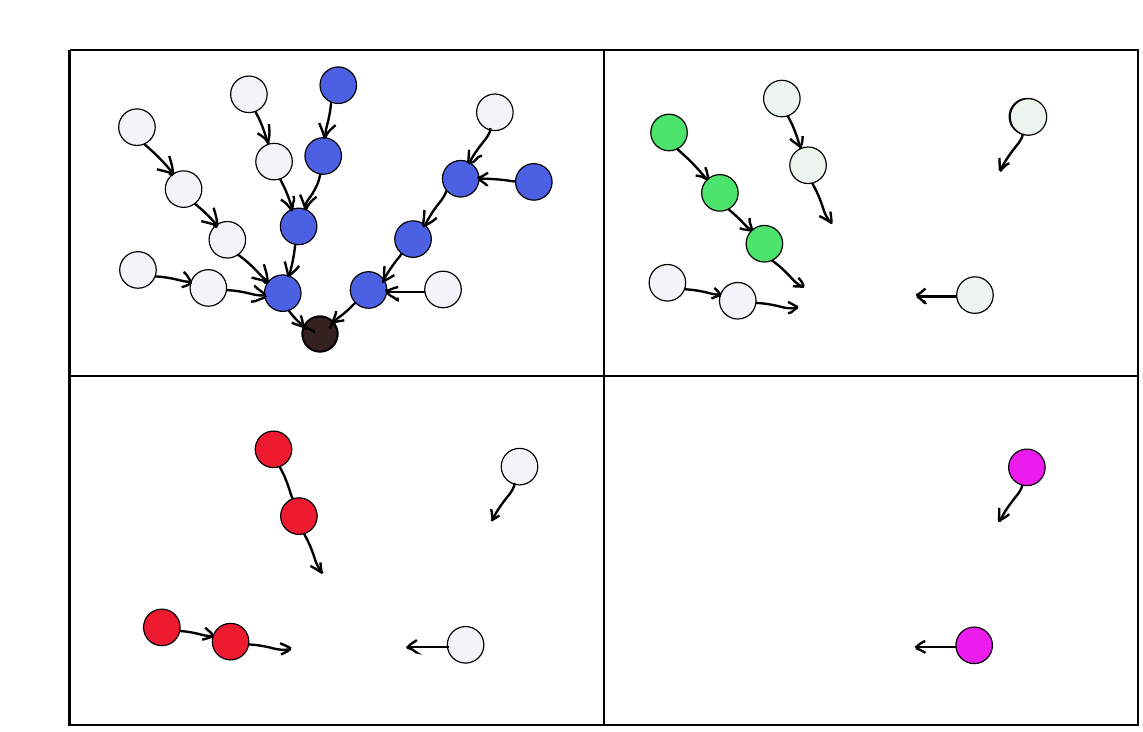}
	
	\caption{Branch identification III.}\label{15-dav}
	
      \end{figure}

      In Fig.~\ref{15-dav}. First square represents the first iteration.  There are 3 paths of length 4 sharing the left edge incident to the root and two paths of length 4 sharing 3 edges. The rightmost path of each group is identified as a 4-branch and colored blue. The second iteration identifies one 3-branch in green; the third iteration identify two 2-branches and the forth iteration identifies two 1-branches. Putting back the colored branches to their original position we obtain the last picture of Fig.~\ref{14-dav}


 \FloatBarrier 

\subsection{Tree-induced soliton decomposition of excursions}
We now take the tree produced by an excursion, as illustrated in Fig.~\ref{excursion-tree}, use any algorithm of \S\ref{ftd} to identify its branches and use the colored tree to identify solitons, as follows.
Put the colored tree back into the excursion and color the diagonal boundaries of the region associated to each node with the color of the node. Each $k$-branch is then associated to $k$ empty and $k$ occupied boxes with the same color; we call those boxes and their content a $k$-soliton*. We use the * to indicate solitons and slots in the tree-induced decomposition. In this case all solitons*  are oriented up, that is, the head of each soliton* is to the left of its tail. See Fig.~\ref{tree-soliton-right}.
\begin{figure}[h!]
	\centering
	\includegraphics[width= .8\textwidth] {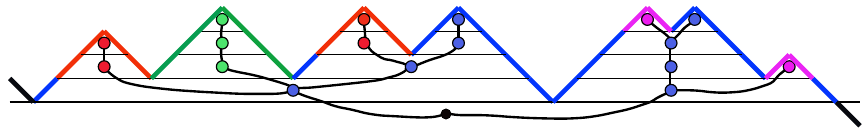}
	\caption{Soliton* decomposition of the excursion of Fig.\ref{excursion} using the branch decomposition of \S\ref{ftd} of the tree associated to the excursion as obtained in Fig.\ref{excursion-tree}, .}\label{tree-soliton-right}
      \end{figure}

      \begin{proposition}[HT and tree decomposition]
        \label{ht2t}
        Given any excursion $\vep$, the HT soliton decomposition of $\vep$ coincides with the tree decomposition of $\vep$.
      \end{proposition}

      \begin{proof}
        This proposition is consequence of Proposition \ref{p11} below, given in terms of the slot diagrams of both objects.
      \end{proof}

\subsection{Slot diagrams of planar trees}
Think each node of a tree as a geometric object. More precisely identify each node with a circumference that is exactly the boundary of the associated colored region like in Figure \ref{14-new-reverse}. Each incident edge to the node is now a segment intersecting the circumference; different edges intersect different points, called incident points. By convention, we assume that there is a segment incident to the root from below. The arcs of the circumference with extremes in the incident points and with no incident point in the interior are called \emph{slots*}. We will describe a procedure to attach new branches to slots*. We use the same symbol $*$ for the solitons of the previous section and slots here since there is a direct correspondence between the solitons* and the slot* diagram for the branches of the tree.





We say that a node of a tree has $k$ \emph{generations} if it is colored in the iteration number $k$ of the algorithm \emph{Branch identification II}. This is equivalent to say that the maximal path from the node to a leaf, moving always in the opposite direction with respect to the root, has $k$ nodes, including the node and the leaf.

\emph{Slots identification of trees I}

Consider a colored tree with maximal branch of size $m$. Declare the whole circumference of the root of the tree as the $m$-slot* number 0; recall there is an incident edge to this node from below.  Attach the $m$-branches to the unique $m$-slot*. Proceed then iteratively for $k<m$.
Assume that the tree has no $\ell$-branches for $\ell\le k$ and call a slot* $\tts$ a \emph{$k$-slot*} if one of the following conditions hold (a) $\tts$ belongs to the root, (b) $\tts$ belongs to a node with more than $k$ generations, (c) $\tts$ belongs to a node with $k$ generations and all path with $k$ nodes containing a leaf incident to the node, is incident to the right of  $\tts$. $k$-slots* are numbered from left to right, starting with $k$-slot* $0$ at the left side of the node associated to the record. See Fig.~\ref{14-new-reverse}.
\begin{figure}[h!]
  \begin{center}
	\includegraphics[width=.7 \textwidth] {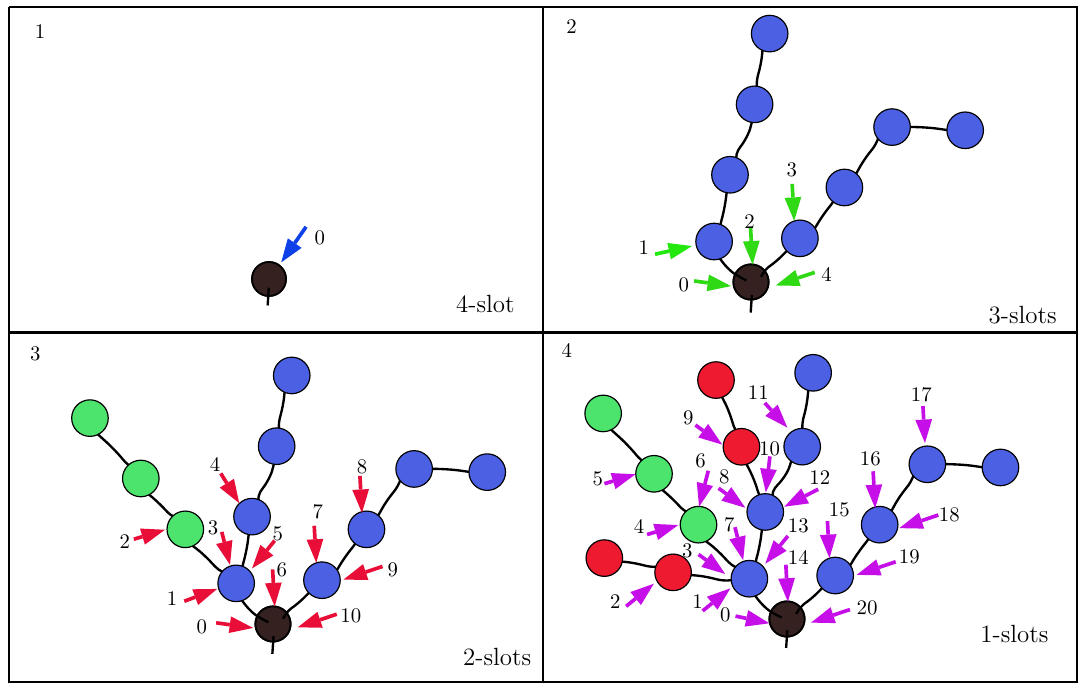}
	\caption{Slot identification. Upper-left: for $m=4$ there is a unique 4-slot* in the root. Upper-right: attaching two 4-branches to this slot we identify five 3-slots*. Attaching one 3-branch to 3-slot 1, identify eleven 2-slots* and finally attaching two 2-branches to 2-slots* 1 and 4, we identify 21 $1$-slots*. To complete the tree in Fig.~\ref{tree-soliton-right} we have to attach two 1-branches (not in this picture). }\label{14-new-reverse}
        \end{center}
      \end{figure}

      %

	
	
	


      \emph{Slot diagram of a tree}

      The slot diagram of the tree is a collection of vectors
      \begin{align}
        x^*=\big((x^*_k(0),\dots,x^*_k(s^*_k-1)):k=1,\dots,m\big)
      \end{align}
      where $m$ is the length of the longest path in the tree and
      \begin{align}
        s^*_m&=1 \text{ and for $k=m,\dots,1$ iterate:}:\nn\\
        x^*_k(i) &= \hbox{number of $k$-branches attached to $k$-slot* number $i$}, \quad i=0,\dots,s^*_k-1\nn\\
        n^*_k&=\sum_{i=0}^{s^*_k-1}x^*_k(i)\\
        s^*_{k-1} &= 1 + \sum_{\ell=k}^m 2(\ell-k)n^*_\ell \label{sk1}
      \end{align}
      In particular the slot* diagram of Fig.\ref{14-new-reverse} is given by $m=4$ and
      \begin{align}
        (s^*_1,s^*_2,s^*_3,s^*_4) &= (21,11,5,1)\nn\\
        x^*_4&= (2)\nn\\
        x^*_3&= (0,1,0,0,0)\label{slot-d}\\
        x^*_2&=(0,1,0,0,1,0, 0,0,0,0,0)\nn\\
        x^*_4&=(0,0,0,0,0,0, 0,0,0,0,0, 0,0,0,0,0, 0,1,0,1,0)\nn
      \end{align}
      using the slot enumeration in Fig.\ref{14-new-reverse}. We illustrate this slot diagram in Fig.~\ref{14-2}.
      \begin{figure}[h!]
	
	\centering
	
	\includegraphics[width=.28 \textwidth] {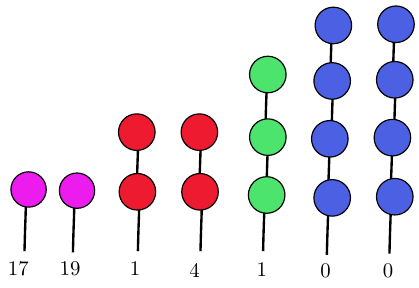}
	
	\caption{Slot* diagram \eqref{slot-d}. There are 2 1-branches attached to 1-slots* number 17 and 19, 2 2-branches attached to 2-slots* 1 and 4; 1 3-branch attached to 3-slot* number 1 and 2 4-branches attached to 4-slot* number 0.}\label{14-2}
      \end{figure}

      \emph{Slots identification of trees II}

     A reverse way to find the slot* diagram of a colored tree with identified slots* is the following. Remove the $1$-branches keeping track of the 1-slot* index each branch was attached to. Assume we have removed the $\ell$-branches for $\ell<k$. Then, remove the $k$-branches keeping track of the $k$-slot* number associated to each removed $k$-branch. The slot* diagram associated to the tree consists on the removed branches and its associated slots* number. See Fig.\ref{14-new}.	
      \begin{figure}[h!]
  \begin{center}
	\includegraphics[width=.7 \textwidth] {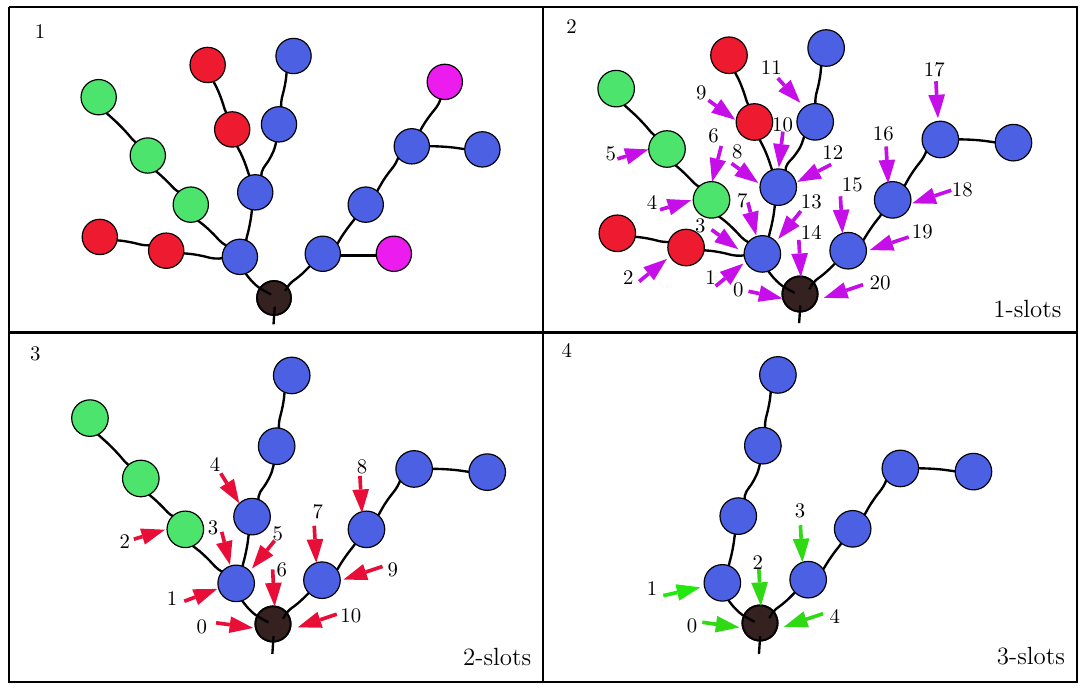}
	\caption{Slot identification of trees II. }\label{14-new}
        \end{center}
      \end{figure}

      Fig.~\ref{14-new}.  Upper-left: a colored tree. Upper-right: erasing 1-branches in the tree, we identify and enumerate $1$-slots*. Lower-left, erasing 1-branches and 2-branches, we identify and enumerate 2-slots*. Lower-right: in a tree with 4-branches we identify and enumerate 3-slots*. The node associated to the record, in black, has one $3$-slot* for each arc.
	
	
	
	
	

	

      \subsection{From paths to trees} We illustrate now the reverse operation. Start with the slot diagram obtained in Fig.\ref{14-2}. Put the root. Let $m$ be the biggest size of the branches in the slot diagram. Attach the $m$-branches to the root. Then successively for $k=m-1,\dots, 1$ attach the $k$ branches to the associated $k$-slot in the tree. The result is illustrated in Fig.\ref{14-new} looking at it backwards: In rectangle 4 we attach 2 4-branches to 4-slot 0 and indicate the place and number of each 3-slot; in rectangle 3 we attach one 3-branch to 3-slot number 1 and so on.

      \begin{proposition}
        \label{p11}
        Given a finite excursion $\vep$ we have $x^{\diamond}[\vep]= x^*[\vep]$.
      \end{proposition}

      \begin{proof}[Sketch proof]
       We give a sketch of the proof showing the basic idea.  Consider an arbitrary slot diagram $x$. We are going to show that the excursion $\vep$ characterized by $x^{\diamond}[\vep]=x$ and the excursion $\vep'$ characterized by $x^*[\vep']=x$ are the same, i.e. $\vep=\vep'$. This implies the statement of the Proposition. Recall that we have constructed the excursion associated to $x^{\diamond}[\vep]$ iteratively in \S\ref{appsoli} gluing one after the other some special triangles. We just showed instead that to construct $x^*[\vep']$ we have to glue recursively the branches like the ones in Fig.~\ref{14-2} glued in Fig.~\ref{14-new} (recall that the gluing procedure has to be followed in the reverse order).

       Since $x$ is the same, both procedures deal with the same number of $k$-triangles and $k$-branches to be attached to the same slots.
       The proof is therefore based on the correspondence between the two different procedures once we fix the basic correspondence of Fig.~\ref{tria-branch} between the two basic building blocks.
       \begin{figure}[h!]
       	
       	\centering
       	
       	\includegraphics[width=.5\textwidth] {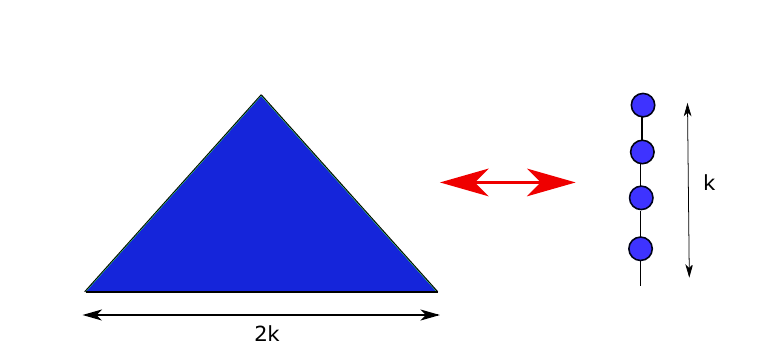}
       	
       	\caption{The correspondence between a triangle and a branch in the two different constructions. Here $k=4$. The excursion associated to each basic building block is the same and coincides with the diagonal boundary of the triangle.}\label{tria-branch}
       	
       \end{figure}
Considering the example of \S\ref{appsoli} we show in Fig.~\ref{similar-tria} the construction of the tree associated to the excursion
$\vep'$ such that $x^*[\vep']=x$ where $x$ is the slot diagram \eqref{scottex}.
This Figure has to be compared with Fig.~\ref{slor-ex} where we constructed the excursion $\vep$ such that $x^{\diamond}[\vep]=x$ where $x$ is again \eqref{scottex}. In Fig.~\eqref{similar-tria} for simplicity we draw just the slots* useful for the attachments. Looking carefully in parallel to the two construction the reader can see that at each step the excursion is the same and the allocations of the slots is again the same. A long formal proof could be given following this strategy.
\begin{figure}[h!]
	
	\centering
	
	\includegraphics[trim = 0 12mm 0 5mm, clip, width=.7\textwidth] {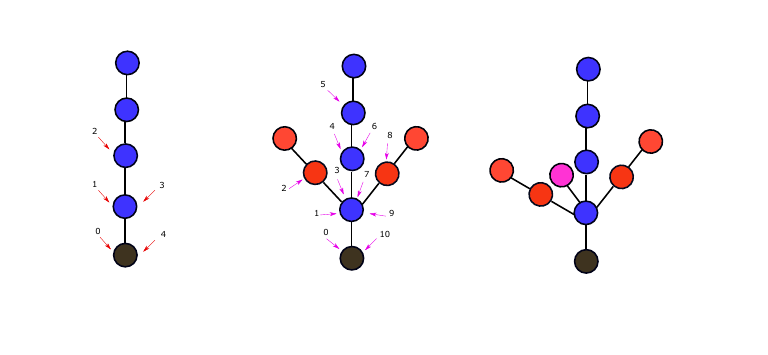}
	
	\caption{The construction of the excursion of Fig.~\ref{slor-ex} using branches instead of triangles. On the left the unique 4-branch attached to the root with the location of the 2-slots* (there are no 3-branches in this case). In the middle the tree after attaching 2-branches to slot* number 1 and slot* number 3, with the location of the 1-slots*. On the right the final tree after attaching the 1-branch to 1-slot* number 3. }	 \label{similar-tria}
	
      \end{figure}
      See also Fig.~\ref{19-ht-tree} for illustration.
 \end{proof}
\begin{figure}[h!]
	
	\centering
	
	\includegraphics[width=\textwidth] {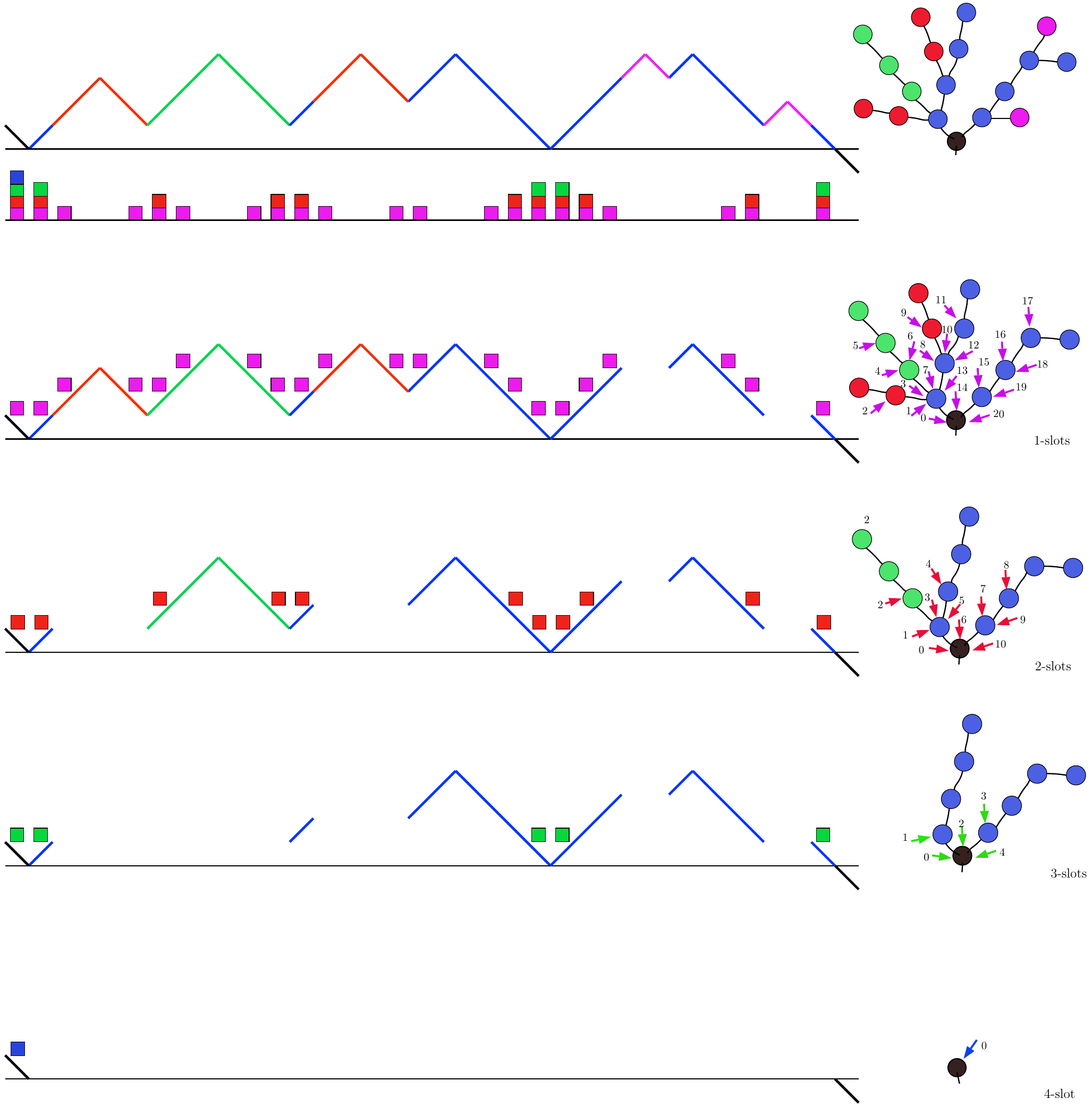}
	
\caption{First line: on the left, the HT soliton decomposition of Fig.~\ref{16-ht+ts} and the localization of the slots$^\diamond$; on the  right, branch decomposition of the tree produced by this excursion from Fig.~\ref{tree-soliton-right} and \ref{14-new}. Following lines: checking that the slot$^\diamond$ localization on the excursion and slot* localization on the tree are the same. In each line we have erased the solitons$^\diamond$/branches smaller than $k$ and show the position of the $k$-slots$^\diamond$/slots*; $k=1,2,3,4$. To see the slot$^\diamond$/slot*  number a soliton$^\diamond$/branch is attached to, look for a square/arrow of the same color in the line below. \label{19-ht-tree}}
 \end{figure}

As a byproduct of the correspondence between planar trees and slot* diagrams we can count the number of planar trees
that have a fixed number of branches. This corresponds to count the number of slot*
diagrams when the numbers $n_{k}^*$ are fixed. For each level $k$ we need to
arrange $n_{k}^*$ branches in $s_k^*$ available slots* and this can be done
in $\binom{n_{k}^*+s_k^*-1}{n_{k}^*}$ different ways. Since this can be done independently on each level we have therefore that the numbers of planar trees
having $n_{k}^*$ branches of length $k$ is given by
\begin{equation}\label{cou}
\prod_{k=1}^M\binom{n_{k}^*+s_k^*-1}{n_{k}^*}=
\prod_{k=1}^M\binom{n_{k}^*+2\sum_{j=k+1}^M(j-k)n_{j}^*}{n_{k}^*}\,,
\end{equation}
where we used \eqref{sk1}.

\FloatBarrier 
\section{Soliton distribution}
\label{sec5}
We report in \S\ref{ss51} a family of distributions on the set of excursions proposed by the authors \cite{fg18} based on the slot decomposition of the excursions. In particular, the slot diagram of the excursion of a random walk satisfies that given the $m$-components for $m>k$, the distribution of the $k$-component is a vector of independent Geometric random variables; the size of the vector is a function of the bigger components.  As a consequence, we obtain that the distribution of the $k$-branches of the tree associated to the excursion of the random walk given the $m$-branches, for $m>k$, is a vector of independent geometric random variables.

Theorem \ref{t6} considers a random ball configuration consisting on iid Bernoulli of parameter $\lambda<\frac12$, conditioned to have a record at the origin and shows that their components are independent and that the $k$-component consists of iid geometric random variables.

Since the measure is given in terms of the number of solitons and slots of the excursion, and those numbers are the same in all the slot diagrams we have introduced, we just work with a generic slot diagram.

\subsection{A distribution on the set of excursions}
\label{ss51}
Let $n_k(\vep)$ be the number of $k$-solitons in the excursion $\vep$ and for $\alpha=(\alpha_k)_{k\ge 1}\in[0,1)^\N$ define
\begin{align}
  \label{aaa}
  Z_{\alpha}:=\textstyle{\sum_{\vep\in\cE} \prod_{k\ge1} \alpha_k^{n_k(\vep)}},
\end{align}
with the convention $0^0=1$. Define
\begin{align}
  \label{eq:cA1}
    \cA&:= \{\alpha \in [0,1)^\N: Z_\alpha<\infty\}
\end{align}
This set has a complex structure since the expression \eqref{aaa} is difficult to handle.
For $\alpha \in \cA$ define the probability measure $\nu_\alpha$ on $\cE$ by
\begin{align}
  \label{nusola}
  \nu_\alpha (\vep):= \frac1{Z_{\alpha}}\textstyle{\prod_{k\ge1} \alpha_k^{n_k(\vep)}}\,.
\end{align}
For $q\in (0,1]^\N$ define the operator $A: q\mapsto \alpha$ by
\begin{align}
  \label{eq:Aq1}
  \alpha_1:= (1-q_1);\qquad \alpha_k&:= (1-q_k) \textstyle{\prod_{j=1}^{k-1}}q_j^{2(k-j)},\quad \hbox{for }k\ge 2.
\end{align}
Reciprocally, define the operator
 $Q: \alpha\mapsto q$ by
\begin{align}
 q_1&:=1-\alpha_1\quad\hbox{and iteratively, } \label{ppc33}\\
  q_k&:=1-\frac{\alpha_k}{\prod_{j=1}^{k-1}q_j^{2(k-j)}}\,, \qquad k\ge 2.\label{ppc34}
\end{align}
Let
\begin{align}
  \label{cQ1}
\cQ&:= \{q\in (0,1]^\N: \textstyle{\sum_{k\ge1}} (1-q_k)<\infty\}.
\end{align}
The next results gives an expression of $\nu_\alpha(\vep)$ in terms of the slot diagram of $\vep$.
\begin{theorem}[Ferrari and Gabrielli \cite{fg18}] \label{fg1}\
     \begin{enumerate}
   \item[(a)] Let $q\in\cQ$, $\alpha=Aq$ and $\nu_\alpha$ given by \eqref{nusola}. Then, $\alpha\in \cA$ and
\begin{align}
  \label{ss24}
    \nu_\alpha(\vep) = \textstyle{\prod_{k\ge 1}}(1-q_k)^{n_k}\, q_k^{s_k}
\end{align}
where $n_k$ and  $s_k$ are the number of $k$-solitons, respectively $k$-slots, of $\vep$.

 \item[(b)]  The map $A:\cQ\to\cA$ is a bijection with $Q=A^{-1}$.
\end{enumerate}
\end{theorem}
The proof of (a) given below shows that if $q\in\cQ$ then $Aq\in\cA$ with $Z_{Aq}= (\prod_{k\ge1}q_k)^ {-1}$. On the other hand, to complete the proof of (b) it suffices to show that $Q\alpha\in\cQ$. The proof of this fact is more involved and can be found in \cite{fg18}.

If we denote $x_k^\infty =(x_k, x_{k+1}, \dots)$, the expression \eqref{ss24} is equivalent to the following (with the convention $q_0:=0$ to take care of the empty excursion).
  \begin{align}
    \nu_\alpha\left(M=m\right)& = (1-q_m)\,\textstyle{\prod_{\ell>m}}q_\ell,    \quad m\ge 0,\label{ss36}\\[2mm]
    \nu_\alpha\big(x_m(0)\big| M=m\big)&= (1-q_m)^{x_m(0)-1} q_m,  \label{ss37}\\[2mm]
    \nu_\alpha\big(x_k\big| x_{k+1}^{\infty}\big)&= (1-q_k)^{n_k} q_k^{s_k}, \label{ss38}
  \end{align}
where we abuse notation writing $x_m$ as ``the set of excursions $\vep$ whose $m$-component in $x[\vep]$ is $x_m$'', and so on. Recall that $n_k$ is the number of $k$-solitons of $x$ and $s_k$ is the number of $k$-slots of $x$, a function of $ x_{k+1}^{\infty}$.

Formulas \eqref{ss36} to \eqref{ss38} give a recipe to construct the slot diagram of a random excursion with law $\nu_\alpha$: first choose a maximal soliton-size $m$ with probability \eqref{ss36} and use \eqref{ss37} to determine the number of maximal solitons $x_m(0)$ (a Geometric$(q_m)$ random variable conditioned to be strictly positive). Then we use \eqref{ss38} to construct iteratively the lower components. In particular, \eqref{ss38} says that under  the measure $\nu_\alpha$ and conditioned on $x_{k+1}^{\infty}$, the variables $\left(x_k(0), \dots x_k(s_k-1)\right)$ are i.i.d. Geometric$(q_k)$.

  \begin{proof}[Proof of Theorem \ref{fg1} (a)]
      Using formula \eqref{sk}, we have
      \begin{align}
        \textstyle{\prod_{k\ge 1}}(1-q_k)^{n_k}\, q_k^{s_k}
&= \textstyle{\prod_{k\ge 1}}(1-q_k)^{n_k} \,q_k^{1+ \sum_{\ell>k} 2(\ell-k) n_\ell}\\[2mm]
        &=\Bigl(\textstyle{\prod_{n\ge1}}q_n\Bigr)\,
          \textstyle{\prod_{k\ge1}}
          \Big[(1-q_k)\textstyle{\prod_{j=1}^{k-1}} \,q_\ell^{2(k-j)}\Big]^{n_{k}}\\[2mm]
        &= \Bigl(\textstyle{\prod_{n\ge1}}q_n\Bigr)\, \textstyle{\prod_{k\ge1}}\alpha_k^{n_k}, \qquad\hbox{denoting }\alpha:= Aq\\
        &= \nu_\alpha(\vep),
      \end{align}
      because $Z_\alpha =\big( \prod_{n\ge1}q_n\big)^{-1}<\infty$ since $q\in \cQ$.
    \end{proof}

\subsection{Branch distribution of the random walk excursion tree} For $\lambda\le \frac12$ define $\alpha=\alpha(\lambda)$ by
  \begin{align}\label{aQ0}
      \alpha_k := \left(\lambda(1-\lambda)\right)^k.
  \end{align}
Then  $\alpha(\lambda) \in\cA$  and $\nu_{\alpha(\lambda)}$ is the law of the excursion of a simple random walk that has probability $\lambda$ to jump up and $1-\lambda$ to jump down. A computation using the Catalan numbers shows that
\begin{align}
  \label{zal}
    Z_{\alpha(\lambda)} = \frac1{1-\lambda}.
  \end{align}
On the other hand, the probability that the random walk perform a fixed excursion with length $2n$ is $\lambda^n(1-\lambda)^{n+1}$, where the extra $(1-\lambda)$ is the probability that the walk jumps down after the $2n$ steps of the excursion. This gives \eqref{zal} with no computations.

  In terms of the branches of the tree associated to the excursion, one chooses the size of the largest branch $m$ of the tree with \eqref{ss36} and use \eqref{ss37} to decide how many maximal branches are attached to the root of the tree. Then identify the $(m-1)$ slots and proceed iteratively using \eqref{ss38} to attach the branches of lower size. Given the branches of size bigger than $k$ already present in the tree, the number of $k$-branches per $k$-slot $\left(x_k(0), \dots x_k(s_k-1)\right)$ are i.i.d. Geometric$(q_k)$ given iteratively by
  \begin{align}
    \label{eq:1}
    q_1 = 1-\lambda(1-\lambda),\quad
    q_k = 1-\frac{(\lambda(1-\lambda))^k}{\textstyle{\prod_{j=1}^{k-1}} \,q_j^{2(k-j)}},\quad k\ge2.
  \end{align}

\subsubsection{Geometric branching processes} Let $\rho$ be a probability measure on $\mathbb N\cup \{0\}$. A branching process with offspring distribution $\rho$
is a random growing tree defined as follows. Let $\left(X_i^j\right)_{i,j\in \mathbb N}$ be a double indexed sequence of i.i.d. random variables having law $\rho$.

At initial time zero the tree is constituted by one single vertex, the root. At time 1
there are $X_1^1$ individuals on the first generation, all of them are generated by the root and are drawn as vertices connected to the root. Give to them an arbitrary order from left to right embedding the tree on a plane. At time 2 each individual of the first generation produces independently a number of new vertices with distribution $\rho$. More precisely the number of vertices produced
by the individual number $i$ of the first generation is $X_2^i$. Every such new vertex is connected by an edge to the parent vertex of the previous generation with an arbitrary order from left to right given by the embedding.
Continue iteratively in this way with $X_j^i$ being the number of vertices of the generation $j$ produced by the individual number $i$ of the generation $j-1$.

If $\sum_{k=0}^{+\infty}k\rho(k)<1$ the branching process is called \emph{subcritical} and the above procedure produces a.e. a finite random planar tree. See \cite{MR942038} for more details, references, and the relation with the law of the corresponding excursions.

\smallskip
Let us consider the case when $\rho$ is the law of a geometric random variable Geometric($1-\lambda$), i.e. $\rho(k)=\mathbb P(X^i_j=k)=(1-\lambda)\lambda^k$, $k=0,1,\dots$.
In this case the probability of any given finite tree is given by
$(1-\lambda)^{|V|}\lambda^{|V|-1}$, where $|V|$ is the number of vertices, included the root. Since $|V|-1=2n$ that is the length of the corresponding excursion (the correspondence is described in section \ref{sected}), we have that
the law of the excursion
coincides with the law of the excursion of a simple random walk having probability $\lambda$ to jump up and probability $1-\lambda$ of jumping down (see \cite{MR942038} for more details).

\smallskip
Using the result discussed in this paper we obtain therefore an alternative procedure of construction of a geometric branching process using independent but not identically distributed geometric random variables.

Consider the parameters $(q_k)_{k\in \mathbb N}$ defined as in \eqref{eq:1}.
The law of the maximal generation $M$ of the branching process is given by the right hand side of \eqref{ss36}, i.e.
$$\mathbb P(M=m)=(1-q_m)\prod_{l>m}q_l\,, \qquad m\geq 0\,.$$

Once the maximal generation has been fixed we attach, directly to the root, a number of maximal branches os size $M=m$ according to the distribution given on the right hand side of \eqref{ss37}, i.e. a Geometric($q_m$) conditioned to be positive. Se for example Figure \ref{14-new-reverse} where $m=4$ and we attach two 4 branches directly to the root in the frame number 2.

We proceed  now iteratively. Suppose that all the branches of size bigger than $k$ have been attached. Consider all the k-slots* of the tree and attach
to all of then a Geometric($q_k$) number of $k$-branches (see for example frame 3 of Figure \ref{14-new-reverse} where we attach 3-branches and frame 4 where we attach 2-branches).

The final random tree obtained this way is a branching process with offspring law $\rho$ given by Geometric($1-\lambda$).

\penalty 10000
  \subsection{Soliton decomposition of product measures in $\{0,1\}^\Z$}

\paragraph{Forest of trees associated to configurations with infinitely many balls}

Consider a configuration $\eta$ (with possibly infinitely many balls) and assume the walk $\xi=W\eta$ has a record at the origin and all records, that is $r(i,\xi)\in\Z$ for all $i\in\Z$. Let $(\vep^i)_{i\in\Z}$ be the excursion decomposition of $\xi$. Associating to each excursion the corresponding tree, we finish with a forest of trees each associated with an excursion, and sharing the slot diagrams of the excursion. See Fig.~\ref{forest-soliton} for the trees associated to the ball configuration in Fig.~\ref{cerchi-q}.
\begin{figure}[h!]
	
	\centering
	
	\includegraphics{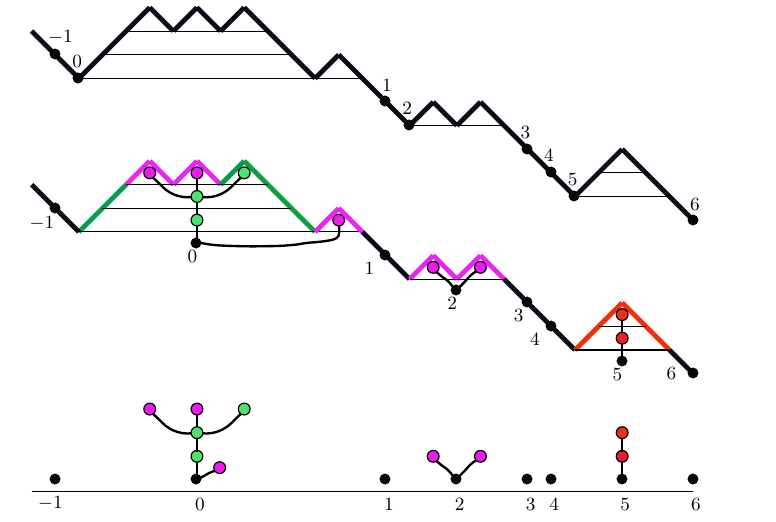}
	
	\caption{Up: Walk representation of the ball configuration of Fig.~\ref{cerchi-q}, with records represented by a black dot and labeled from $-1$ to 6. Middle: black dots representing records with nonempty excursions have been displaced to facilitate the picture. Each excursion tree has been decomposed into branches, with the corresponding colors. Down: the forest representing this piece of walk.}\label{forest-soliton}
	
      \end{figure}

\paragraph{Soliton decomposition of configurations with infinitely many balls}  For the same walk $\xi$ with excursion components $(\vep^i)_{i\in\Z}$, consider $(x^i)_{i\in\Z}$, the set of slot diagrams associated to those excursions. Recall $\vep^i$ is the excursion between Record $i$ and Record $i+1$.

  We define the vector $\zeta\in \big((\N\cup\{0\})^\Z\big)^{\N}$ obtained by concatenation of the $k$-components of $x^i$ as follows:
  \begin{align}
       S_k^0 &:= 0,\quad\hbox{ for all }k\ge 1\nn\\
    S_k^{i'}-S_k^{i} &:= s_k^i+\dots+ s_k^{i'-1}+ i'-i, \hbox{ for }i<i',\, k\ge 1, \label{z87}\\
    \zeta_k(S_k^i+ j) &:= x_k^{i}(j), \quad j\in\{0,\dots,s^i_k-1\},\, k\ge 1.\nn
  \end{align}
  The components of $\zeta$ are $\zeta_k(j)\in \mathbb N\cup\{0\}$ with $k\in \mathbb N$
and $j\in \mathbb Z$.
  For example, Fig.~\ref{forest-soliton} contains a piece of $\xi$ between Record $-1$ and Record 6. The excursions $\vep^{-1},\vep^1,\vep^3,\vep^4$ are empty, so $s^i_k= 1$ for $k\ge1$ and $x^i_k(0)= 0$ for $k\ge 0$ and $i=-1,1,3,4$.
The corresponding slot diagrams are
  \begin{align}
    \label{x43}
    &x^{-1}=x^1=x^3=x^4=\emptyset\nn\\
    &x^0_3= (1), x^0_2=(0,0,0),  x^0_1=(0,0,2,0,1) \nn\\
    &x^2_1=(2)\nn\\
    &x^5_2=(1), x^5_1= (0,0,0)
  \end{align}
  So that the maximal soliton number in the slot diagram $i$ is $m^i= 0$ for $i\in\{-1,1,3,4\}$,   $m^0=3$, $m^2=2$ and $m^5=2$.

  \emph{Young diagram}. To better explain graphically the definitions \eqref{z87} and construct the piece of configuration $\zeta$ corresponding to the above excursions, we associate a Young diagram to each slot diagram, as follows: for each soliton size $k$ on the slot diagram $x$ pile one row of size $s_k$ for $k\le m$ and one row of length 1 for all $k>m$. We finish with an infinite column at slot 0 and all $k$-slots of the same number piled on the same column. Taking the vertical coordinate as $k$ and the horizontal coordinate as $j$, in box $(j,k)$ put $x_k(j)$. Fig.~\ref{forest-diagram} shows the Young diagrams corresponding to \eqref{x43}.
\begin{figure}[h!]
	
	\centering
	
	\includegraphics[width=.8 \textwidth]{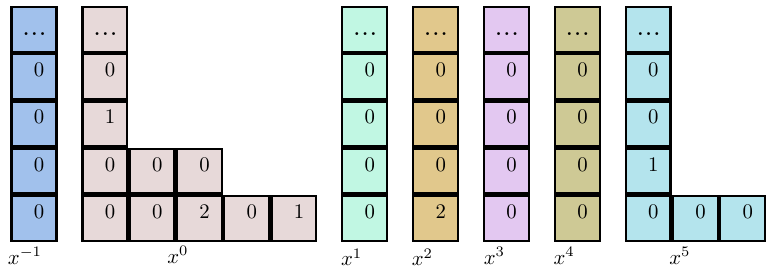}
	
	\caption{Young shape of slot diagrams of excursions $-1$ to 5 of Fig.~\ref{forest-soliton}. Dots at top of columns mean that this is an infinite column of zeroes from that point up.}\label{forest-diagram}
	
\end{figure}

Once we have the slot diagrams of the excursions of $\xi$ as decorated Young tableaux, to obtain $\zeta$ it suffices to glue the rows of the same height into a unique row justified by column 0, as in the Fig.~\ref{forest-decomposition}.
\begin{figure}[h!]
	
	\centering
	
	\includegraphics[width=.68 \textwidth]{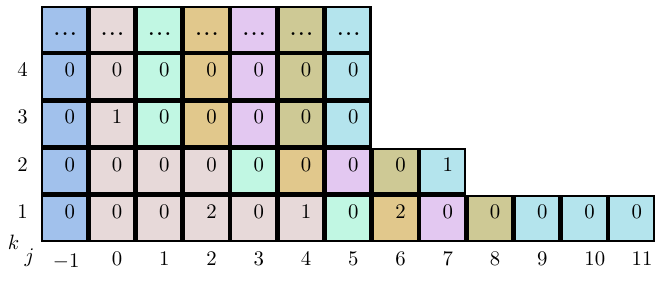}
	
	\caption{Justify the slots diagrams with the column at the 0-slot. The result is the piece of configuration $\zeta$ produced by the excursions $-1$ to $5$. In the vertical coordinate the $k$-component, in the horizontal coordinate, the slot number. For example $\zeta_2(7) = 1$.}\label{forest-decomposition}
	
      \end{figure}
  We call $\cA^+$ the set of $\alpha$ such that the mean excursion size under $\nu_\alpha$ is finite:
  \begin{equation}
  \mathcal A^+:=\left\{\alpha\,:\, \textstyle{\sum_{k\ge1}2k\rho_k(\alpha)}<+\infty \right\}\,.\label{a+}
  \end{equation}
  By definition we have $\mathcal A^+\subseteq \mathcal A$. We define also
  $$
  \mathcal Q^+:=\left\{q\,:\, \textstyle{\sum_{k\ge1}}\,k(1-q_k)<+\infty\right\}\,.\label{sumq+}
  $$
      \begin{theorem}[From independent solitons to independent iid geometrics \cite{fg18}]
        \label{t6}
   If $\alpha\in\cA^+$ and $(\vep^i)_{i\in\Z}$ are iid excursions with distribution $\nu_\alpha$, then $(\zeta_k)_{k\in\Z}\in(\N\cup\{0\})^\Z$, as defined in \eqref{z87} is a family of independent configurations and for each $k$, $ (\zeta_k(j))_{j\in\Z}$ are iid random variables with distribution Geometric($q_k$), where $A^{-1}\alpha=q\in \mathcal Q^+$.
   \end{theorem}

   \begin{proof}
     A part of the proof is a direct consequence of \eqref{ss36}-\eqref{ss37}-\eqref{ss38}. For the proof of the remaining statements see \cite{fg18}.
   \end{proof}

   \section*{Acknowledgments} We thank Leo Rolla for many fruitful discussions and Jean Fran\c cois Le Gall for pointing out relevant references on excursion trees.

This project started when PAF was visiting GSSI at L'Aquila en 2016. He thanks the hospitality and support.
Part of this project was developed during the stay of the authors at the Institut Henri  Poincar\'e - Centre \'Emile Borel during the trimester \emph{Stochastic Dynamics Out  of Equilibrium}. We thank this institution for hospitality and support.

\addcontentsline{toc}{section}{References}

\bibliographystyle{acm}

\begin{thebibliography}{10}

\bibitem{cbs19}
{\sc Cao, X., Bulchandani, V.~B., and Spohn, H.}
\newblock The {GGE} averaged currents of the classical {Toda} chain.
\newblock {\em arXiv:1905.04548\/} (2019).

\bibitem{CKST}
{\sc Croydon, D.~A., Kato, T., Sasada, M., and Tsujimoto, S.}
\newblock Dynamics of the box-ball system with random initial conditions via
  {Pitman's} transformation.
\newblock {\em arXiv:1806.02147\/} (2018).

\bibitem{croydon2019invariant}
{\sc Croydon, D.~A., and Sasada, M.}
\newblock Invariant measures for the box-ball system based on stationary
  {M}arkov chains and periodic {G}ibbs measures.
\newblock {\em arXiv:1905.00186\/} (2019).

\bibitem{MR0370775}
{\sc Dwass, M.}
\newblock Branching processes in simple random walk.
\newblock {\em Proc. Amer. Math. Soc. 51\/} (1975), 270--274.

\bibitem{evans}
{\sc Evans, S.~N.}
\newblock {\em Probability and real trees}, vol.~1920 of {\em Lecture Notes in
  Mathematics}.
\newblock Springer, Berlin, 2008.
\newblock Lectures from the 35th Summer School on Probability Theory held in
  Saint-Flour, July 6--23, 2005.

\bibitem{fg18}
{\sc Ferrari, P.~A., and Gabrielli, D.}
\newblock {BBS} invariant measures with independent soliton components.
\newblock {\em arXiv:1812.02437\/} (2018).

\bibitem{FNRW}
{\sc Ferrari, P.~A., Nguyen, C., Rolla, L., and Wang, M.}
\newblock Soliton decomposition of the box-ball system.
\newblock {\em arXiv:1806.02798\/} (2018).

\bibitem{IKT}
{\sc Inoue, R., Kuniba, A., and Takagi, T.}
\newblock Integrable structure of box--ball systems: crystal, {Bethe} ansatz,
  ultradiscretization and tropical geometry.
\newblock {\em Journal of Physics A: Mathematical and Theoretical 45}, 7
  (2012), 073001.

\bibitem{KTZ}
{\sc Kato, T., Tsujimoto, S., and Zuk, A.}
\newblock Spectral analysis of transition operators, automata groups and
  translation in {BBS}.
\newblock {\em Communications in Mathematical Physics 350}, 1 (2017), 205--229.

\bibitem{MR0290475}
{\sc Kawazu, K., and Watanabe, S.}
\newblock Branching processes with immigration and related limit theorems.
\newblock {\em Teor. Verojatnost. i Primenen. 16\/} (1971), 34--51.

\bibitem{2018arXiv180808074K}
{\sc {Kuniba}, A., and {Lyu}, H.}
\newblock {One-sided scaling limit of multicolor box-ball system}.
\newblock {\em arXiv:1808.08074\/} (Aug 2018), arXiv:1808.08074.

\bibitem{MR942038}
{\sc Le~Gall, J.-F.}
\newblock Une approche \'{e}l\'{e}mentaire des th\'{e}or\`emes de
  d\'{e}composition de {W}illiams.
\newblock In {\em S\'{e}minaire de {P}robabilit\'{e}s, {XX}, 1984/85},
  vol.~1204 of {\em Lecture Notes in Math.} Springer, Berlin, 1986,
  pp.~447--464.

\bibitem{LG}
{\sc Le~Gall, J.-F.}
\newblock Random trees and applications.
\newblock {\em Probab. Surv. 2\/} (2005), 245--311.

\bibitem{LLP}
{\sc Levine, L., Lyu, H., and Pike, J.}
\newblock Double jump phase transition in a random soliton cellular automaton.
\newblock {\em arXiv:1706.05621\/} (2017).

\bibitem{MIT}
{\sc Mada, J., Idzumi, M., and Tokihiro, T.}
\newblock The exact correspondence between conserved quantities of a periodic
  box-ball system and string solutions of the {Bethe} ansatz equations.
\newblock {\em Journal of mathematical physics 47}, 5 (2006), 053507.

\bibitem{C}
{\sc Stanley, R.~P.}
\newblock {\em Catalan numbers}.
\newblock Cambridge University Press, New York, 2015.

\bibitem{TS}
{\sc Takahashi, D., and Satsuma, J.}
\newblock A soliton cellular automaton.
\newblock {\em Journal of The Physical Society of Japan 59}, 10 (1990),
  3514--3519.

\bibitem{TTMS}
{\sc Tokihiro, T., Takahashi, D., Matsukidaira, J., and Satsuma, J.}
\newblock From soliton equations to integrable cellular automata through a
  limiting procedure.
\newblock {\em Physical Review Letters 76}, 18 (1996), 3247.

\bibitem{YYT}
{\sc Yoshihara, D., Yura, F., and Tokihiro, T.}
\newblock Fundamental cycle of a periodic box-ball system.
\newblock {\em J. Phys. A 36}, 1 (2003), 99--121.

\bibitem{YT}
{\sc Yura, F., and Tokihiro, T.}
\newblock On a periodic soliton cellular automaton.
\newblock {\em Journal of Physics A: Mathematical and General 35}, 16 (2002),
  3787.

\end{thebibliography}

\end{document}